\documentclass{amsart}
\usepackage{amssymb}
\usepackage{amsfonts}

\newtheorem{theorem}{Theorem}[section]
\theoremstyle{plain}

\newtheorem{claim}[theorem]{Claim}

\newtheorem{definition}[theorem]{Definition}

\newtheorem{lemma}[theorem]{Lemma}
\newtheorem{notation}[theorem]{Notation}

\newtheorem{proposition}[theorem]{Proposition}
\newtheorem{remark}[theorem]{Remark}

\numberwithin{equation}{section}
\input{tcilatex}
\theoremstyle{remark}
\newtheorem*{acknowledgement}{Acknowledgement}

\begin{document}
\title[Orbit Injections Equivalence and C*-algebras]{C*-algebraic
characterization of Bounded Orbit Injection Equivalence for Minimal Free
Cantor Systems}
\author{Fr\'{e}d\'{e}ric Latr\'{e}moli\`{e}re}
\address{Department of Mathematics\\
University of Denver}
\email{frederic@math.du.edu}
\urladdr{http://www.math.du.edu/\symbol{126}frederic}
\author{Nicholas Ormes}
\address{Department of Mathematics\\
University of Denver}
\email{normes@math.du.edu}
\urladdr{http://www.math.du.edu/\symbol{126}normes}
\date{March 9, 2009}

\begin{abstract}
Bounded orbit injection equivalence is an equivalence relation defined on
minimal free Cantor systems which is a candidate to generalize flip Kakutani
equivalence to actions of the Abelian free groups on more than one
generator. This paper characterizes bounded orbit injection equivalence in
terms of a mild strengthening of Rieffel-Morita equivalence of the
associated C*-crossed-product algebras. Moreover, we construct an ordered
group which is an invariant for bounded orbit injection equivalence, and
does not agrees with the $K_{0}$ group of the associated C*-crossed-product
in general. This new invariant allows us to find sufficient conditions to
strengthen bounded orbit injection equivalence to orbit equivalence and
strong orbit equivalence.
\end{abstract}

\maketitle

\section{\protect\bigskip Introduction}

This paper establishes a characterization of bounded orbit injection
equivalence, as introduced in \cite{Lightwood07} by S. Lightwood and the
second author, in terms of a strengthened form of Rieffel-Morita equivalence
between C*-crossed-products. For minimal $\mathbb{Z}$-actions of the Cantor
set, bounded orbit injection equivalence is equivalent to flip-Kakutani
equivalence, i.e., the equivalence relation generated by Kakutani
equivalence and time reversal. Bounded orbit injection equivalence is a
generalization of flip-Kakutani equivalence which applies to actions of $%
\mathbb{Z}^{d}$ where time reversal is not a well-defined concept. While
Giordano, Putnam and Skau have shown in \cite[Theorem 2.6]{Putnam95}\ that
Kakutani strong orbit equivalence is characterized by Rieffel-Morita
equivalence of the C*-crossed-products, the C*-algebraic picture of Kakutani
equivalence, and more generally bounded orbit injection equivalence, is the
main new result of this article and, informally, can be described as a form
of Rieffel-Morita equivalence where moreover the space on which the action
occurs is remembered. We begin our paper with an introduction of the
concepts we will use and the framework for our characterization. We then
establish our characterization in the next section. In the last section we
apply our results to derive a sufficient condition for bounded orbit
injections to give rise to strong orbit equivalence. This condition involves
an ordered group which is a direct $\mathbb{Z}^{d}$ analog of the group used
in \cite[Theorem 2.6]{Putnam95} when $d=1$, but in contrast to that case, it
is not the $K_{0}$-group of the C*-algebra when $d>1$. Giordano, Matui,
Putnam and Skau have recently shown that this group modulo the infinitesmal
subgroup characterizes minimal $\mathbb{Z}^{d}$ Cantor systems up to orbit
equivalence \cite{Putnam09}.

A triple $\left( X,\varphi ,\mathbb{Z}^{d}\right) $ is a dynamical system
(on a compact space) when $X$ is a compact space and $\varphi $ is an action
of $\mathbb{Z}^{d}$ on $X$ by homeomorphisms. A Cantor system $\left(
X,\varphi ,\mathbb{Z}^{d}\right) $ is a dynamical system where $X$ is a
Cantor set. Moreover, if $\varphi $ is a free action, then $\left( X,\varphi
,\mathbb{Z}^{d}\right) $ will be called free as well, and if every point in $%
X$ has a dense orbit for the action $\varphi $ then $\left( X,\varphi ,%
\mathbb{Z}^{d}\right) $ will be called minimal.

Let $\left( X,\varphi ,\mathbb{Z}^{d}\right) $ and $(Y,\psi ,\mathbb{Z}^{d})$
be two free minimal Cantor systems for some positive integer $d$. We wish to
investigate when $\varphi $ and $\psi $ are equivalent in some dynamically
meaningful way. A natural concept of equivalence is given by conjugacy: $%
\varphi $ and $\psi $ are conjugate when there exists a homeomorphism $%
h:X\rightarrow Y$ such that $h\circ \varphi =\psi \circ h$. However,
classification up to conjugacy is a very complex problem, and it has proven
fruitful to study weaker form of equivalences with more tractable
invariants. A fundamental example of such an equivalence is orbit
equivalence \cite{Putnam95}:\ the actions $\varphi $ and $\psi $ of $\mathbb{%
Z}$ are orbit equivalent when there exists a homeomorphism $h:X\rightarrow Y$
and two maps $n:X\rightarrow \mathbb{Z}$ and $m:Y\rightarrow \mathbb{Z}$
such that for all $x\in X$ and $y\in Y$ we have $h\circ \varphi (x)=\psi
^{n(x)}\circ h(x)$ and $\varphi ^{m(y)}\circ h^{-1}(y)=h^{-1}\circ \psi (y)$%
. In other words, $\varphi $ and $\psi $ are orbit equivalent if and only if
there exists a homeomorphism $h:X\rightarrow Y$ which maps orbits to orbits.
Following \cite{Putnam95}, the actions $\varphi $ and $\psi $ are strongly
orbit equivalent when the maps $m$ and $n$ are discontinuous at one point at
most. Giordano, Putnam and Skau proved in \cite[Theorem 2.6]{Putnam95}\ that
the C*-crossed-product algebra of the two minimal free actions $\varphi $
and $\psi $ of $\mathbb{Z}$ on a Cantor set are *-isomorphic if and only if $%
\varphi $ and $\psi $ are strongly orbit equivalent.

\bigskip The C*-algebra of an action $\alpha $ of $\mathbb{Z}^{d}$ on a
compact set $X$ is defined as the universal C*-algebra $C(X)\rtimes _{\alpha
}\mathbb{Z}^{d}$ generated by $C(X)$ and unitary operators $U_{\alpha }^{z}$
for all $z\in \mathbb{Z}^{d}$ subject to the relations $U_{\alpha
}^{z}f\left( U_{\alpha }^{z}\right) ^{\ast }=f\circ \alpha ^{-z}$ and $%
U_{\alpha }^{z}U_{\alpha }^{z^{\prime }}=U_{\alpha }^{z+z^{\prime }}$ for
all $f\in C(X)$ and $z,z^{\prime }\in \mathbb{Z}^{d}$ with $U_{\alpha
}^{0}=1 $. C*-crossed-products were introduced in Zeller-Meier in \cite%
{Zeller-Meier68} and have a rich and complex structure as C*-algebras \cite%
{Zeller-Meier68}\cite{Pedersen79} \cite{Williams07}, whose connection with
the defining action is not always clear. It is thus a remarkable fact that
when $\alpha $ is a minimal and free action of $\mathbb{Z}$ on a Cantor set $%
X$, the crossed-product $C(X)\rtimes _{\alpha }\mathbb{Z}$ is a complete
invariant for strong orbit equivalence. Moreover, in this case, the
C*-crossed-products are inductive limits of so called circle algebras and
are fully classified up to *-isomorphism by their $K$ groups (including the
order on $K_{0}$ and some distinguished elements in each group $K_{0}$ and $%
K_{1}$) , as shown by Elliott in \cite{Elliott93}. Thus, the ordered $K_{0}$
group and its order unit of the C*-crossed-product for such actions form a
complete invariant of strong orbit equivalence.

\bigskip Yet, in general, proving that two C*-algebras are *-isomorphic is
nontrivial, and again some weaker but interesting form of equivalence have
been introduced to help with this problem. Two C*-algebras $A$ and $B$ are
Rieffel-Morita equivalent when, informally, their categories of
non-degenerate representations on Hilbert spaces are equivalent. More
formally, two C*-algebras $A$ and $B$ are Rieffel-Morita equivalent when
there exists a full $B$ Hilbert module $M\ $such that the C*-algebra of
compact adjoinable operators on $M$ is *-isomorphic to $A$ \cite[Theorem
4.26 p. 164]{GraciaVarilly}\cite{Rieffel74b}. In particular, when the
C*-algebras $A$ and $B$ are simple, if there exists a *-isomorphism $\varphi
:A\rightarrow pBp$ with $p$ some nonzero projection in $B$ then $A$ and $B$
are Morita equivalent. In \cite[Theorem 2.6]{Putnam95}, Giordano, Putnam and
Skau established that $C(X)\rtimes _{\varphi }\mathbb{Z}$ and $C(Y)\rtimes
_{\psi }\mathbb{Z}$ are Rieffel-Morita equivalent if and only if $\varphi $
and $\psi $ are Kakutani strong orbit equivalent. Kakutani strong orbit
equivalence is defined in terms of induced systems:\ if $A$ is a clopen
subset of a Cantor set $Z$ and $\alpha $ is a minimal free action of $%
\mathbb{Z}$ on $Z$ then $\alpha ^{A}$ is the action of $\mathbb{Z}$ on $A$
defined by first return times of $\alpha $ to $A$, and is called a derived
system of $\alpha $. Now, $\varphi $ and $\psi $ are strongly orbit Kakutani
equivalent if there exists a free minimal action $\alpha $ of $\mathbb{Z}$
on some Cantor set $Z$ such that both $\varphi $ and $\psi $ are conjugated
to derived systems of $\alpha $. The systems $\varphi $ and $\psi $ are
Kakutani strong orbit equivalent when they are strongly orbit equivalent to
derived systems of $\alpha $. Thus once again, a dynamical concept such as
Kakutani strong orbit equivalence is characterized by a standard concept of
C*-algebra theory --- Rieffel-Morita equivalence. It is thus natural to
investigate analogous C*-algebraic descriptions of other forms of
equivalence between minimal free Cantor systems.

\bigskip Two minimal free systems $\left( X,\varphi ,\mathbb{Z}\right) $ and 
$\left( Y,\psi ,\mathbb{Z}\right) $ are flip-Kakutani equivalent if $\left(
X,\varphi ,\mathbb{Z}\right) $ is Kakutani equivalent to either $\left(
Y,\varphi ,\mathbb{Z}\right) $ or its time reversed system $\left( Y,\varphi
\circ \sigma ,\mathbb{Z}\right) $ where $\sigma :z\in \mathbb{Z}\mapsto -z$.
Now, $\sigma $ is the only automorphism of $\mathbb{Z}$ besides the
identity, but the situation is more complicated for $\mathbb{Z}^{d}$ in
general. To generalize the notion of flip-Kakutani equivalence, S. Lightwood
and the second author introduced in \cite{Lightwood07} the notion of orbit
injection equivalence between minimal free actions of $\mathbb{Z}^{d}$ on
the Cantor set. In general, for an action $\alpha $ of $\mathbb{Z}^{d}$ on a
set $Z$ we denote the image of $z\in \mathbb{Z}^{d}$ by $\alpha ^{z}$. We
recall from \cite{Lightwood07}:

\begin{definition}
\label{OrbitInjection}Let $\left( X,\varphi ,\mathbb{Z}^{d}\right) $ and $%
\left( Y,\psi ,\mathbb{Z}^{d}\right) $ be two free dynamical systems. Then a
map $\theta :X\rightarrow Y$ is an \emph{orbit injection} from $\left(
X,\varphi ,\mathbb{Z}^{d}\right) $ to $\left( Y,\psi ,\mathbb{Z}^{d}\right) $
when $\theta $ is a continuous open injection such that for all $x,y\in X$
we have: 
\begin{equation}
\exists w\in \mathbb{Z}^{d}\text{ such that }\varphi ^{w}(x)=y\ \text{if and
only if }\exists v\in \mathbb{Z}^{d}\text{ such that }\psi ^{v}\circ \theta
(x)=\theta (y)\text{.}  \label{oicocycle}
\end{equation}%
The \emph{cocycle} for an orbit injection $\theta $ is the function $\eta
:X\times \mathbb{Z}^{d}\rightarrow \mathbb{Z}^{d}$ defined by $\psi ^{\eta
\left( x,w\right) }\theta \left( x\right) =\theta \left( \varphi
^{w}x\right) $. The orbit injection $\theta $ is called \emph{bounded }if
the cocycle $\eta $ is continuous.
\end{definition}

\bigskip We observe that, as the actions $\varphi $ and $\psi $ are free,
Identity (\ref{oicocycle}) uniquely defines the cocycle $\eta $. Moreover,
the range of an orbit injection is clopen, since it is assumed open by
definition and it is the continuous image of a compact set so it is closed
as well.

\begin{remark}
In \cite{Lightwood07}, bounded orbit injections are not required to have
open range, and it is shown instead that for $d=2$ the range of bounded
orbit injections has nonempty interior. In general, given two minimal free
Cantor systems $\left( X,\varphi ,\mathbb{Z}^{d}\right) $ and $\left( Y,\psi
,\mathbb{Z}^{d}\right) ,$ the range of a bounded orbit injection $\theta $
from $\left( X,\varphi ,\mathbb{Z}^{d}\right) $ to $\left( Y,\psi ,\mathbb{Z}%
^{d}\right) $ has nonempty interior if and only if it is open (or
equivalently $\theta $ is an open map). Indeed, assume $X^{\prime }=\theta
(X)$ contains an open subset $U$. Let $y\in X^{\prime }$. By minimality
there exists $z\in \mathbb{Z}^{d}$ and $\omega \in U$ such that $\psi
^{z}(\omega )=y$. Then there exists a unique $x\in X$ and $h\in \mathbb{Z}%
^{d}$ such that $\theta (x)=\omega $ and $\theta \left( \varphi ^{-h}\left(
x\right) \right) =y$ since $\theta $ is an orbit injection. Now, since the
cocycle $\eta $ associated to $\theta $ is continuous, the subset $\Omega
=\left( \eta \left( \cdot ,z\right) \right) ^{-1}\left( \left\{ -h\right\}
\right) $ of $X$ is open. Let $V=\Omega \cap \theta ^{-1}(U)$ which is open
in $X$ as well. Now, since $X$ is compact and $\theta $ is a continuous
injection, $\theta $ is a homeomorphism from $X$ onto $X^{\prime }$ for its
relative topology in $Y$. Thus in particular, $\theta (V)$ is relatively
open in $X^{\prime },$ i.e. there exists an open subset $\Upsilon $ of $Y$
such that $\theta (V)=\Upsilon \cap X^{\prime }$. Yet $\theta (V)\subseteq
U\subseteq X^{\prime }$ and $U$ is open in $Y$ so $\theta (V)=\Upsilon \cap
U $ is open in $Y$. Moreover, by construction, $\psi ^{z}\left( \theta
\left( V\right) \right) =\theta \left( \varphi ^{-h}\left( V\right) \right) $
so $\psi ^{z}\left( \theta \left( V\right) \right) $ is an open subset of $Y$
(since $\psi ^{z}$ is a homeomorphism) and is contained in $\theta (X)$. Yet 
$y\in \psi ^{z}\left( \theta \left( V\right) \right) $, so $\theta (X)$ is a
neighborhood of each of its points and is thus open in $Y$. The converse is
immediate.
\end{remark}

\begin{definition}
\label{BOIE}Let $\left( X,\varphi ,\mathbb{Z}^{d}\right) $ and $\left(
Y,\psi ,\mathbb{Z}^{d}\right) $ be two free minimal Cantor systems. Then $%
\left( X,\varphi ,\mathbb{Z}^{d}\right) $ and $\left( Y,\psi ,\mathbb{Z}%
^{d}\right) $ are\emph{\ bounded orbit injection equivalent} when there
exists a minimal free Cantor system $\left( Z,\alpha ,\mathbb{Z}^{d}\right) $
with two bounded orbit injections $\theta _{\varphi }$ and $\theta _{\psi }$
from, respectively, $\left( X,\varphi ,\mathbb{Z}^{d}\right) $ and $\left(
Y,\psi ,\mathbb{Z}^{d}\right) $ to $\left( Z,\alpha ,\mathbb{Z}^{d}\right) $.
\end{definition}

\bigskip The fact that bounded orbit injection equivalence is reflexive and
symmetric is obvious, and transitivity was proven indirectly in \cite%
{Lightwood07} for $d=2$ where bounded orbit injection equivalence is proven
to be the same relation as suspension equivalence. In the last section of
this paper, we will give a proof that bounded orbit injection equivalence is
transitive for any $d\in \mathbb{N}\setminus \left\{ 0\right\} $ and thus is
indeed an equivalence relation.

\bigskip This paper establishes that the systems $\left( X,\varphi ,\mathbb{Z%
}^{d}\right) $ and $\left( Y,\psi ,\mathbb{Z}^{d}\right) $ are bounded orbit
injection equivalent if and only if both $C(X)\rtimes _{\varphi }\mathbb{Z}%
^{d}$ and $C(Y)\rtimes _{\psi }\mathbb{Z}^{d}$ can be embedded as corners in 
$C(Z)\rtimes _{\alpha }\mathbb{Z}^{d}$ for some minimal free Cantor system $%
\left( Z,\alpha ,\mathbb{Z}^{d}\right) $, and the images by these embeddings
of both $C(X)$ and $C(Y)$ are subalgebras of $C(Z)$. We note that this
result is new even in the case of actions of $\mathbb{Z}$ as a
characterization of flip-Kakutani of minimal free Cantor systems. The proof
of this first result occupies our first section. It partly relies upon
techniques inspired by \cite[Theorem 2.6]{Putnam95}, as well as other tools
such as spectral decomposition of C*-crossed-products.

\bigskip We show in Theorem \ref{thm:iso} that for a minimal Cantor system $%
\left( X,\varphi ,\mathbb{Z}^{d}\right) $ the ordered group $\left( G\left(
\varphi \right) ,G\left( \varphi \right) _{+}\right) $ where

\begin{eqnarray*}
G\left( \varphi \right) &=&C\left( X,\mathbb{Z}\right) /\langle f-f\varphi
^{v}:v\in \mathbb{Z}^{d}\rangle \\
G\left( \varphi \right) _{+} &=&\left\{ \left[ f\right] :f\left( x\right)
\geq 0\text{ for all }x\in X\right\}
\end{eqnarray*}%
is an invariant of bounded orbit injection equivalence. In the case when $%
\varphi $ is a $\mathbb{Z}$-action, by the Pimsner-Voiculescu six term exact
sequence \cite{Blackadar98} this is the $K_{0}$-group of the C*-algebra $%
C(X)\rtimes _{\varphi }\mathbb{Z}$. We show that $\left( G\left( \varphi
\right) ,G\left( \varphi \right) _{+}\right) $ shares many properties with
simple dimension groups. However, as is shown in \cite{Matui08}, this group
may contain torsion even when $d=2$. The connection between the ordered
group $\left( G\left( \varphi \right) ,G\left( \varphi \right) _{+}\right) $
and the C*-algebra $C(X)\rtimes _{\varphi }\mathbb{Z}^{d}$ for $d>1$ is
unclear (see discussion in \cite{PutnamSurvey}\cite{Phillips05}). Thus while
it is not particularly surprising that it is an invariant of bounded orbit
injection equivalence, it is interesting that by the aforementioned results
it is an invariant of this strengthened notion of Reiffel-Morita equivalence
for $C(X)\rtimes _{\varphi }\mathbb{Z}^{d}$.

\bigskip We show that given bounded orbit injections $\theta _{1}$, $\theta
_{2}$ from two minimal Cantor systems $\left( X,\varphi ,\mathbb{Z}%
^{d}\right) $ and $\left( Y,\psi ,\mathbb{Z}^{d}\right) $ into a third $%
\left( Z,\alpha ,\mathbb{Z}^{d}\right) $, the question of whether the system 
$\left( Z,\alpha ,\mathbb{Z}^{d}\right) $ may be omitted turns on the nature
of the isomorphisms $h_{\theta _{1}}$ and $h_{\theta _{2}}$ from $\left(
G\left( \varphi \right) ,G\left( \varphi \right) _{+}\right) $ and $\left(
G\left( \psi \right) ,G\left( \psi \right) _{+}\right) $ into $\left(
G\left( \alpha \right) ,G\left( \alpha \right) _{+}\right) $ induced by the
orbit injections, or if one prefers, the nature of the isomorphism $%
h=h_{\theta _{2}}^{-1}h_{\theta _{1}}:\left( G\left( \varphi \right)
,G\left( \varphi \right) _{+}\right) \rightarrow \left( G\left( \psi \right)
,G\left( \psi \right) _{+}\right) $. Specifically, letting $\left[ 1_{X}%
\right] $, $\left[ 1_{Y}\right] $ represent the equivalence classes of the
constant functions $1$ on the spaces $X,Y$, if $h\left[ 1_{X}\right] =\left[
1_{Y}\right] $ then we show that there is a bounded orbit equivalence from $%
\left( X,\varphi ,\mathbb{Z}^{d}\right) $ to $\left( Y,\psi ,\mathbb{Z}%
^{d}\right) $. If $h\left[ 1_{X}\right] -\left[ 1_{Y}\right] $ is an
infinitesimal element then we show that the orbit injections can be modified
to a (not necessarily bounded)\ orbit equivalence from $\left( X,\varphi ,%
\mathbb{Z}^{d}\right) $ to $\left( Y,\psi ,\mathbb{Z}^{d}\right) $. Finally,
if $\left[ 1_{Y}\right] -h\left[ 1_{X}\right] \in G\left( \psi \right) _{+}$
then there is a bounded orbit injection from $X$ to $Y$.

\bigskip As is the case in \cite{Lightwood07, Putnam09}, partitions
associated with tilings are a key tool for proving results about $\mathbb{Z}%
^{d}$-dynamics and we use Voronoi tilings for our results about the group $%
\left( G\left( \varphi \right) ,G\left( \varphi \right) _{+}\right) $. In
particular, the notion of a Voronoi-Rohlin partition serves as a $\mathbb{Z}%
^{d}$ version of the tower partitions used in \cite[Theorem 2.6]{Putnam95}.

\begin{acknowledgement}
We wish to thank Alvaro Arias, Thierry Giordano, Hiroki Matui, Ian Putnam,
and Christian Skau for helpful discussions.
\end{acknowledgement}

\section{C*-algebraic Characterization of Orbit Injection Equivalence}

\bigskip This section establishes the characterization of bounded orbit
injection equivalence in terms of Rieffel-Morita equivalence of
C*-algebras.\ Given a dynamical system $\left( X,\varphi ,\mathbb{Z}%
^{d}\right) $ on a compact space $X$, the C*-crossed-product $C(X)\rtimes
_{\varphi }\mathbb{Z}^{d}$ is the universal C*-algebra generated by a copy
of $C(X)$ and a family $\left( U_{\varphi }^{z}\right) _{z\in \mathbb{Z}%
^{d}} $ of unitary operators satisfying the relations $U_{\varphi
}^{z}fU_{\varphi }^{-z}=f\circ \varphi ^{-z}$, $U_{\varphi }^{z}U_{\varphi
}^{z^{\prime }}=U_{\varphi }^{z+z^{\prime }}$ and $U_{\varphi }^{0}=1$ for
all $z,z^{\prime }\in \mathbb{Z}^{d}$ and $f\in C(X)$. The construction of
this C*-algebra is detailed in \cite{Zeller-Meier68},\cite{Pedersen79} and 
\cite{Williams07}. Note that in this paper, we shall follow the convention
that $\varphi ^{z}$ is the homeomorphism given by the action $\varphi $ on $%
X $ at $z\in \mathbb{Z}^{d}$, and the associated unitary in the
crossed-product $C(X)\rtimes _{\varphi }\mathbb{Z}^{d}$, which we will call
the canonical unitary for $\varphi $ at $z$, is denoted by $U_{\varphi }^{z}$%
.

\bigskip This section generalizes methods developed in \cite{Putnam95} to
the case of $\mathbb{Z}^{d}$ actions. An important tool in \cite{Putnam95}
is the description of normalizers of the C*-subalgebra $C(X)$ in the
crossed-product $C(X)\rtimes \mathbb{Z}$. We shall also need such a
description, but in addition we need to make sure that the normalizers we
will encounter form a group isomorphic to $\mathbb{Z}^{d}$ --- something
which is automatic when $d=1$ but needs some efforts for the general case we
study in this paper. In general, the generalization from actions of $\mathbb{%
Z}$ to actions of $\mathbb{Z}^{d}$ involve various technical points which we
emphasize in the proofs of this section.

\bigskip As a first step, we shall establish the following proposition,
which provides a mean to embed a crossed-product into another given a
bounded orbit injection:

\begin{proposition}
\label{FundationL2}Let $X$ and $Z$ be Cantor sets. Let $\left( X,\varphi ,%
\mathbb{Z}^{d}\right) $ and $\left( Z,\psi ,\mathbb{Z}^{d}\right) $ be two
free dynamical systems and $\theta :X\rightarrow Z$ a bounded orbit
injection with unique cocycle $\eta $. Then the projections defined for $%
z,h\in \mathbb{Z}^{d}$ and $y\in Z$ by:%
\begin{equation}
p_{h}^{z}(y)=\left\{ 
\begin{array}{c}
1\text{ if }\exists x,x^{\prime }\in X\ \ \ y=\theta (x)\text{ and }%
x=\varphi ^{z}(x^{\prime })\text{ and }y=\psi ^{h}\left( \theta \left(
x^{\prime }\right) \right) \text{,} \\ 
0\text{ otherwise.}%
\end{array}%
\right.  \label{ProjDef}
\end{equation}%
satisfy:

\begin{itemize}
\item For $z,h\in \mathbb{Z}^{d}$, the projection $p_{h}^{z}$ is in $C(Z)$,
hence in $C(Z)\rtimes _{\psi }\mathbb{Z}^{d}$,

\item For $z\in \mathbb{Z}^{d}$ and $h\not=h^{\prime }\in \mathbb{Z}^{d}$ we
have $p_{h}^{z}p_{h^{\prime }}^{z}=0$,

\item For $z\in \mathbb{Z}^{d}$ the set $\left\{ h\in \mathbb{Z}%
^{d}:p_{h}^{z}\not=0\right\} $ is finite,

\item The set $\theta (X)$ is clopen in $Z$,

\item If $p$ is the projection on $\theta (X)$ i.e. is the multiplication
operator by the indicator of $\theta (X)$ in $Z$, then for all $z\in \mathbb{%
Z}^{d}$ we have:%
\begin{equation*}
\sum_{h\in \mathbb{Z}^{d}}p_{h}^{z}=p\text{,}
\end{equation*}

\item For $z,z^{\prime },h\in \mathbb{Z}^{d}$ we have:%
\begin{equation*}
p_{h}^{z+z^{\prime }}=\sum_{h^{\prime }\in \mathbb{Z}^{d}}p_{h^{\prime
}}^{z}\cdot p_{h-h^{\prime }}^{z^{\prime }}\circ \psi ^{-h^{\prime }}\text{,}
\end{equation*}
\end{itemize}

Consequently, setting for $z\in \mathbb{Z}^{d}$:%
\begin{equation*}
V^{z}=\sum_{h\in \mathbb{Z}^{d}}p_{h}^{z}U_{\psi }^{h}+\left( 1-p\right)
\end{equation*}%
then the map $z\in \mathbb{Z}^{d}\mapsto V^{z}$ is a group isomorphism into
the unitary group in $C\left( Z\right) \rtimes _{\psi }\mathbb{Z}^{d}$.
Moreover, $V^{z}p=pV^{z}$ for all $z\in \mathbb{Z}^{d}$.
\end{proposition}

\begin{proof}
We denote $\theta (X)\ $by $X^{\prime }$. By assumption, $\theta $ is a
continuous bijection from $X$ onto $X^{\prime }$. Since $X$ is compact, $%
\theta $ is in fact a homeomorphism from $X$ onto $X^{\prime }$ for the
relative topology on $X^{\prime }$. Moreover, as the continuous image of a
compact set, $X^{\prime }$ is closed in $Z$. Since by assumption on $\theta $
the set $X^{\prime }$ is also open in $Z$, we conclude that $\theta (X)$ is
clopen in $Z$.

For all $y\in Z$ and $z,h\in \mathbb{Z}^{d}$, we define $p_{h}^{z}(y)$ by
Identity (\ref{ProjDef}) and we note that the support $X_{h}^{z}$ of $%
p_{h}^{z}$ is the image by $\theta $ of set $P_{h}^{z}=\left( \eta \left(
\cdot ,z\right) \right) ^{-1}\left( \left\{ h\right\} \right) $, the latter
being clopen in $X$ since $\eta $ is continuous. Since $X^{\prime }$ is
clopen in $Z$ and $\theta $ is a homeomorphism from $X$ onto $X^{\prime }$
we deduce that $X_{h}^{z}$ is clopen in $Z$. Hence $p_{h}^{z}$ is a
continuous function on $Z$. We now establish the properties of the lemma.

Assume $X_{h}^{z}\cap X_{h^{\prime }}^{z}\not=\emptyset $ for some $%
z,h,h^{\prime }\in \mathbb{Z}^{d}$. Let $y\in X_{h}^{z}\cap X_{h^{\prime
}}^{z}$. Then by definition, there exists $x,x^{\prime },x^{\prime \prime
},x^{\prime \prime \prime }\in X$ such that $y=\theta (x)=\theta (x^{\prime
\prime })$, with $x=\varphi ^{z}(x^{\prime })$ and $x^{\prime \prime
}=\varphi ^{z}(x^{\prime \prime \prime })$ and $y=\psi ^{h}\left( \theta
(x^{\prime })\right) =\psi ^{h^{\prime }}\left( \theta (x^{\prime \prime
\prime })\right) $. Since $\theta $ is injective, $x=x^{\prime \prime }$.
Since $\varphi ^{z}$ is a homeomorphism, we have $x^{\prime }=x^{\prime
\prime \prime }$. Since the action by $\psi $ is free, $h=h^{\prime }$.
Thus, if $h\not=h^{\prime }$ then $X_{h}^{z}\cap X_{h^{\prime
}}^{z}=\emptyset $.

Moreover, let $y\in X^{\prime }$. By definition, there exists (a unique) $%
x\in X$ such that $y=\theta (x)$. Set $x^{\prime }=\varphi ^{-z}(x)$. Since $%
\theta $ is an orbit injection with cocycle $\eta $, we have:%
\begin{equation*}
\psi ^{\eta (x,z)}(\theta (x^{\prime }))=\theta \left( \varphi
^{z}(x^{\prime })\right) =\theta (x)=y
\end{equation*}%
by Definition (\ref{OrbitInjection}). Hence by definition $y\in X_{\eta
(x,z)}^{z}$. Conversely, if $y\in X_{h}^{z}$ then by definition $y\in
X^{\prime }$. Thus for any $z\in \mathbb{Z}^{d}$ the set $X^{\prime }$ is
the union of the clopen sets $X_{h}^{z}$ for $h\in \mathbb{Z}^{d}$ and thus $%
X^{\prime }$ itself is clopen. Furthermore, as $X^{\prime }$ is closed in
the compact set $X$ we conclude that $X^{\prime }$ is compact. Thus for any $%
z\in \mathbb{Z}^{d}$, if $\mathcal{P}_{z}=\left\{ h\in \mathbb{Z}%
^{d}:X_{h}^{z}\not=\emptyset \right\} $ then the family $\left(
X_{h}^{z}\right) _{h\in \mathcal{P}_{z}}$ is a partition of the compact $%
X^{\prime }$ by open subsets, so it is finite. Hence $\mathcal{P}_{z}$,
which equals $\left\{ h\in \mathbb{Z}^{d}:p_{h}^{z}\not=0\right\} $ by
definition, is finite as claimed.

We thus have proven that for a fixed $z\in \mathbb{Z}^{d}$, the projections $%
p_{h}^{z}$ ($h\in \mathbb{Z}^{d}$) are pairwise orthogonal and that $%
\sum_{h\in \mathbb{Z}^{d}}p_{h}^{z}=p$ where $p$ is the projection on $%
X^{\prime }$ i.e. the indicator function of $X^{\prime }$ in $C\left(
Z\right) $ (note that $p\in C(Z)$ since $X^{\prime }$ is clopen in $Z$) --
and where only finitely many terms in the sum are nonzero. We now prove a
natural convolution product relates the projections $p_{h}^{z}$ for varying $%
z\in \mathbb{Z}^{d}$.

Let $z,z^{\prime },h\in \mathbb{Z}^{d}$. Let $y\in Z$. Then $%
p_{h}^{z+z^{\prime }}(y)=1$ if and only if there exists $x,x^{\prime }\in X$
such that:%
\begin{equation}
y=\theta (x)\text{, }x=\varphi ^{z+z^{\prime }}(x^{\prime })\text{ and }%
y=\psi ^{h}\left( \theta \left( x^{\prime }\right) \right) \text{.}
\label{convolution_0}
\end{equation}%
Since $x=\varphi ^{z}\left( \varphi ^{z^{\prime }}(x^{\prime })\right) $, by
assumption on $\theta $, there exists $h^{\prime }=\eta \left( \varphi
^{z^{\prime }}(x^{\prime }),z\right) \in \mathbb{Z}^{d}$ such that $y=\psi
^{h^{\prime }}\left( \theta \left( \varphi ^{z^{\prime }}\left( x^{\prime
}\right) \right) \right) $. In summary, for some $h^{\prime }\in \mathbb{Z}%
^{d}$:%
\begin{equation}
y=\psi ^{h^{\prime }}\left( \theta \left( \varphi ^{z^{\prime }}\left(
x^{\prime }\right) \right) \right) \text{ and }x=\varphi ^{z}\left( \varphi
^{z^{\prime }}(x^{\prime })\right) \text{.}  \label{convolution_1}
\end{equation}%
Hence by definition, $p_{h^{\prime }}^{z}(y)=1$. Note moreover that $%
h^{\prime }$ is unique, since $p_{h^{\prime }}^{z}(y)=1$ implies that $%
p_{h^{\prime \prime }}^{z}(y)=0$ for all $h^{\prime \prime }\not=h^{\prime }$%
.

Furthermore, since $\psi ^{-h^{\prime }}=\left( \psi ^{h^{\prime }}\right)
^{-1}$ we have from the first equality in (\ref{convolution_1}):%
\begin{equation*}
\psi ^{-h^{\prime }}\left( y\right) =\theta \left( \varphi ^{z^{\prime
}}(x^{\prime })\right)
\end{equation*}%
and from the last equality in (\ref{convolution_0}):%
\begin{equation*}
\psi ^{-h^{\prime }}(y)=\psi ^{h-h^{\prime }}\left( \theta \left( x^{\prime
}\right) \right)
\end{equation*}%
and therefore by definition:%
\begin{equation*}
p_{h-h^{\prime }}^{z^{\prime }}\left( \psi ^{-h^{\prime }}\left( y\right)
\right) =1\text{.}
\end{equation*}%
Hence we obtain the desired formula:%
\begin{equation}
p_{h}^{z+z^{\prime }}(y)=\sum_{h^{\prime }\in \mathbb{Z}^{d}}p_{h^{\prime
}}^{z}(y)\left( p_{h-h^{\prime }}^{z^{\prime }}\circ \psi ^{-h^{\prime
}}\left( y\right) \right) \text{,}  \label{convolution_p}
\end{equation}%
noting that for any $y\in X$ only one term at most in the sum is nonzero.

We can now define for all $z\in \mathbb{Z}^{d}$ the following operator:%
\begin{equation*}
V^{z}=\sum_{h\in \mathbb{Z}^{d}}p_{h}^{z}U_{\psi }^{h}+\left( 1-p\right) \in
C\left( Z\right) \rtimes _{\psi }\mathbb{Z}^{d}\text{,}
\end{equation*}%
where again the sum is over only finitely many nonzero terms.

Now, let $y\in Z\backslash X^{\prime }$ and let $z,h\in \mathbb{Z}^{d}$. We
compute:%
\begin{eqnarray*}
\left( p_{h}^{z}U_{\psi }^{h}\left( 1-p\right) \right) ^{\ast }\left(
p_{h}^{z}U_{\psi }^{h}\left( 1-p\right) \right) &=&\left( 1-p\right) \left(
p_{h}^{z}\circ \psi ^{h}\right) (1-p) \\
&=&\left( 1-p\right) \left( p_{h}^{z}\circ \psi ^{h}\right) \text{.}
\end{eqnarray*}%
Yet if $p_{h}^{z}\circ \psi ^{h}(y)=1$ then by definition there exists $%
x,x^{\prime }\in X^{\prime }$ such that $\psi ^{h}(y)=\theta (x)$, $%
x=\varphi ^{z}(x^{\prime })$ and $\psi ^{-h}\left( \psi ^{h}(y)\right)
=\theta (x^{\prime })\in X^{\prime }$. This last equation forces $y\in
X^{\prime }$ and thus $\left( 1-p\right) (y)=0$. Hence: 
\begin{equation}
p_{h}^{z}U_{\psi }^{h}\left( 1-p\right) =0\text{.}  \label{convolution_2}
\end{equation}

By the convolution formula (\ref{convolution_p}), Equality (\ref%
{convolution_2}) and the trivial observation:%
\begin{equation*}
\left( \left( 1-p\right) p_{h}^{z}\right) (y)=0
\end{equation*}%
for all $z,h\in \mathbb{Z}^{d}$ and $y\in Z$, we conclude:%
\begin{eqnarray*}
V^{z}V^{z^{\prime }} &=&\sum_{h\in \mathbb{Z}^{d}}\sum_{h^{\prime }\in 
\mathbb{Z}^{d}}p_{h}^{z}U_{\psi }^{h}p_{h^{\prime }}^{z^{\prime }}U_{\psi
}^{h^{\prime }}+\left( 1-p\right) \text{ yet }U_{\psi }^{h^{\prime
}}=U_{\psi }^{-h}U_{\psi }^{h+h^{\prime }}\text{ so:} \\
V^{z}V^{z^{\prime }} &=&\sum_{h\in \mathbb{Z}^{d}}\sum_{h^{\prime }\in 
\mathbb{Z}^{d}}\left( p_{h}^{z}p_{h^{\prime }}^{z^{\prime }}\circ \psi
^{-h}\right) U_{\psi }^{h+h^{\prime }}+\left( 1-p\right) \\
&=&\sum_{h^{\prime \prime }\in \mathbb{Z}^{d}}\left( \sum_{h\in \mathbb{Z}%
^{d}}p_{h}^{z}p_{h^{\prime \prime }-h}^{z^{\prime }}\circ \psi ^{-h}\right)
U_{\psi }^{h^{\prime \prime }}+\left( 1-p\right) \text{ where }h^{\prime
\prime }=h+h^{\prime } \\
&=&\sum_{h^{\prime \prime }\in \mathbb{Z}^{d}}p_{h^{\prime \prime
}}^{z+z^{\prime }}U_{\psi }^{h^{\prime \prime }}+\left( 1-p\right)
=V^{z+z^{\prime }}\text{.}
\end{eqnarray*}%
Moreover by construction $V^{0}=p+1-p=1$. Hence $z\mapsto V^{z}$ is a
morphism from $\mathbb{Z}^{d}$ into the unitary group of $C\left( Z\right)
\rtimes _{\psi }\mathbb{Z}^{d}$. Last, we note that Identity (\ref%
{convolution_2}) shows that $p_{h}^{z}U_{\psi }^{h}=p_{h}^{z}U_{\psi }^{h}p$
for all $z,h\in \mathbb{Z}^{d}$, which in turn establishes that for all $%
z\in \mathbb{Z}^{d}$ the unitary $V^{z}$ commutes with $p$. Our proposition
is thus proven.
\end{proof}

\bigskip As a first use of Proposition\ (\ref{FundationL2}), we prove that
the existence of a bounded orbit injection implies Rieffel-Morita
equivalence of crossed-products with an additional property, thus
establishing the necessary condition of our characterization of bounded
orbit injection equivalence:

\begin{proposition}
\label{Morita1}Let $\left( X,\varphi ,\mathbb{Z}^{d}\right) $ and $\left(
Z,\psi ,\mathbb{Z}^{d}\right) $ be two minimal free Cantor systems and let $%
\theta :X\rightarrow Z$ be a bounded orbit injection from $\left( X,\varphi ,%
\mathbb{Z}^{d}\right) $ to $\left( Z,\psi ,\mathbb{Z}^{d}\right) $. Then
there exists a *-monomorphism $\alpha :C(X)\rtimes _{\varphi }\mathbb{Z}%
^{d}\longrightarrow C(Z)\rtimes _{\psi }\mathbb{Z}^{d}$ such that $\alpha
(C(X))\subseteq C(Z)$ and whose range is the corner algebra $p\left(
C(Z)\rtimes _{\psi }\mathbb{Z}^{d}\right) p$ where $p=\alpha (1)$.
\end{proposition}

\begin{proof}
Let $f\in C(X)$ and $y\in Z$. We set:%
\begin{equation*}
\pi (f)(y)=\left\{ 
\begin{array}{c}
f(x)\text{ if }y=\theta (x)\text{,} \\ 
0\text{ otherwise.}%
\end{array}%
\right.
\end{equation*}%
First, $\pi (f)$ is a well-defined map from $Z$ to $\mathbb{C}$ since $%
\theta $ is injective. Moreover, it is straightforward to check that $\pi
(f) $ is continuous over $Z$ since the range of $\theta $ is clopen by
Proposition (\ref{FundationL2}). Let $V,p$ and $p_{h}^{z}$ ($z,h\in \mathbb{Z%
}^{d}$) be given by by Lemma\ (\ref{FundationL2}), which applies since $%
\theta $ is a continuous orbit injection and both $\varphi $ and $\psi $ are
free by assumption.\ Note that $p$ is the indicator function of $\theta (X)$
in $Z$, so $p=\pi (1)$. Let now $f\in C(X)$ be given. We wish to check that $%
\left( \pi ,V\right) $ is a covariant representation of $\left( X,\varphi ,%
\mathbb{Z}^{d}\right) $ so we compute:%
\begin{eqnarray*}
V^{z}\pi (f)V^{-z} &=&\left( \sum_{h\in \mathbb{Z}^{d}}p_{h}^{z}U_{\psi
}^{h}+(1-p)\right) \pi (f)\left( \sum_{h^{\prime }\in \mathbb{Z}%
^{d}}p_{h^{\prime }}^{z}U_{\psi }^{h^{\prime }}+(1-p)\right) ^{\ast } \\
&=&\sum_{h\in \mathbb{Z}^{d}}\sum_{h^{\prime }\in \mathbb{Z}%
^{d}}p_{h}^{z}U_{\psi }^{h}\pi (f)U_{\psi }^{-h^{\prime }}p_{h^{\prime }}^{z}%
\text{,}
\end{eqnarray*}%
since $\left( 1-p\right) \pi (f)=\pi (f)\left( 1-p\right) =0$. Now, for $%
h,h^{\prime }\in \mathbb{Z}^{d}$ we have:%
\begin{eqnarray}
p_{h}^{z}U_{\psi }^{h}\pi (f)U_{\psi }^{-h^{\prime }}p_{h^{\prime }}^{z}
&=&p_{h}^{z}\left( \pi (f)\circ \psi ^{-h}\right) U_{\psi }^{h-h^{\prime
}}p_{h^{\prime }}^{z}  \notag \\
&=&\left( \pi (f)\circ \psi ^{-h}\right) \left( p_{h}^{z}\left( p_{h^{\prime
}}^{z}\circ \psi ^{h^{\prime }-h}\right) \right) U_{\psi }^{h-h^{\prime }} 
\notag \\
&=&\left\{ 
\begin{array}{c}
0\text{ if }h\not=h^{\prime } \\ 
\pi (f)\circ \psi ^{-h}\text{ if }h=h^{\prime }\text{.}%
\end{array}%
\right.  \label{Morita_id1}
\end{eqnarray}%
Indeed, suppose that $p_{h}^{z}(y)=1$ for some $h\in \mathbb{Z}^{d}$ and $%
y\in Z$. Then by definition, there exists $x,x^{\prime }\in X$ such that $%
y=\theta (x)=\psi ^{h}\left( \theta (x^{\prime })\right) $ and $x=\varphi
^{z}(x^{\prime })$. Then:%
\begin{equation*}
\psi ^{h^{\prime }-h}(y)=\psi ^{h^{\prime }-h}\left( \psi ^{h}\left( \theta
(x^{\prime })\right) \right) =\psi ^{h^{\prime }}\left( \theta (x^{\prime
})\right) \text{.}
\end{equation*}%
Yet if $p_{h^{\prime }}^{z}(\psi ^{h^{\prime }-h}\left( y\right) )=1$ for
some $h^{\prime }\in \mathbb{Z}^{d}$ then there exists $w,w^{\prime }\in X$
such that $\psi ^{h^{\prime }-h}(y)=\theta (w)=\psi ^{h^{\prime }}(\theta
(w^{\prime }))$ and $w=\varphi ^{z}(w^{\prime })$. Yet then since $\psi
^{h^{\prime }}\circ \theta $ is injective, we conclude that $w^{\prime
}=x^{\prime }$ and thus $x=w$ so $\psi ^{h^{\prime }-h}(y)=y$. Hence as $%
\psi $ is free we conclude that $h=h^{\prime }$. Thus $p_{h}^{z}\left(
p_{h^{\prime }}^{z}\circ \psi ^{h^{\prime }-h}\right) (y)$ is $1$ when $%
h=h^{\prime }$ and $0$ otherwise.

Hence by Identity (\ref{Morita_id1}) we have $V^{z}\pi (f)V^{-z}\in C(Z)$
and moreover by construction, if $x\in X$ and $y=\theta (x)$, since $\theta $
is an orbit injection with cocycle $\eta $:%
\begin{eqnarray*}
V^{z}\pi (f)V^{-z}(y) &=&\left( \sum_{h\in \mathbb{Z}^{d}}p_{h}^{z}\pi
(f)\circ \psi ^{-h}\right) (y) \\
&=&p_{\eta (x,z)}^{z}(y)\pi \left( f\circ \varphi ^{-z}\right) (y) \\
&=&\pi \left( f\circ \varphi ^{-z}\right) (y)\text{.}
\end{eqnarray*}%
On the other hand, if $y\not\in \theta (X)$ then:%
\begin{eqnarray*}
V^{z}\pi (f)V^{-z}(y) &=&\left( \sum_{h\in \mathbb{Z}^{d}}p_{h}^{z}\pi
(f)\circ \psi ^{-h}\right) (y) \\
&=&0 \\
&=&\pi (f\circ \varphi ^{-z})(y)
\end{eqnarray*}%
where the last equality follows from $\pi (g)(Z/X^{\prime })=\left\{
0\right\} $ for all $g\in C(X)$ by construction.

Hence the pair $\left( \pi ,V\right) $ is covariant for $\left( X,\varphi ,%
\mathbb{Z}^{d}\right) $. It is however degenerate, and its integrated
*-morphism is actually valued in the corner $p\left( C(Z)\rtimes _{\psi }%
\mathbb{Z}^{d}\right) p$. It will be convenient to work explicitly in this
corner. We note that for all $f\in C(X)$ we have $\pi (f)\in pC(Z)p=C\left(
\theta (X)\right) $ and $p$ commutes with $V^{z}$ for all $z\in \mathbb{Z}%
^{d}$. Hence, setting $\pi ^{\prime }=p\pi \left( \cdot \right) p$ and $%
\omega ^{g}=pV^{g}p$ for $g\in \mathbb{Z}^{d}$ we define a nondegenerate
covariant pair $\left( \pi ^{\prime },\omega \right) $ valued in $p\left(
C(Z)\rtimes _{\psi }\mathbb{Z}^{d}\right) p$. By \cite[Proposition 2.39]%
{Williams07} there exists a unique *-morphism $\alpha $ such that for all $%
f\in C(X)$ we have $\alpha (f)=\pi ^{\prime }(f)\in pC(Z)p$ and for all $%
z\in \mathbb{Z}^{d}$ we have $\alpha (U_{\varphi }^{z})=\omega ^{z}$. Since $%
(X,\varphi ,\mathbb{Z}^{d})$ is minimal, the morphism $\alpha $ is
injective. We now investigate the range of $\alpha $. First by construction,
the range of $\alpha $ is a C*-subalgebra of $p(C(Z)\rtimes _{\psi }\mathbb{Z%
}^{d})p$ and $\pi ^{\prime }(1)=\alpha (1)=p$. Second, $\pi ^{\prime
}=\alpha _{|C(X)}$ is a *-isomorphism onto $C\left( \theta (X)\right)
\subseteq C(Z)$. Now, let $h\in \mathbb{Z}^{d}$. Set $X_{h}=\left\{ x\in
X^{\prime }:\psi ^{h}(x)\in X^{\prime }\right\} $ and note that $X_{h}$ is
clopen in $X^{\prime }$. Let $q_{h}$ be the indicator function of $X_{h}$ in 
$C(Z)$. Now:%
\begin{equation*}
\left( \left( p-q_{h}\right) U_{\psi }^{h}\left( p-q_{h}\right) \right)
^{\ast }\left( \left( p-q_{h}\right) U_{\psi }^{h}\left( p-q_{h}\right)
\right) =\left( p-q_{h}\right) \cdot \left( p-q_{h}\right) \circ \psi ^{h}%
\text{.}
\end{equation*}%
Note that $p-q_{h}$ is a projection (since $q_{h}$ is a subprojection of $p$%
). Now, if $\left( p-q_{h}\right) (x)=1$ then $x\in \theta (X)$ and $\psi
^{h}(x)\not\in \theta (X)$ and thus $\left( p-q_{h}\right) (\psi ^{h}(x))=0$%
. Hence $\left( p-q_{h}\right) U_{\psi }^{h}\left( p-q_{h}\right) =0$.
Similarly: 
\begin{equation*}
\left( q_{h}U_{\psi }^{h}\left( p-q_{h}\right) \right) ^{\ast }\left(
q_{h}U_{\psi }^{h}\left( p-q_{h}\right) \right) =\left( p-q_{h}\right) \cdot
q_{h}\circ \psi ^{h}=0
\end{equation*}%
and thus $q_{h}U_{\psi }^{h}\left( p-q_{h}\right) =0$. We would prove $%
\left( p-q_{h}\right) U_{\psi }^{h}q_{h}=0$ the same way. We conclude:%
\begin{equation*}
pU_{\psi }^{h}p=q_{h}U_{\psi }^{h}q_{h}\text{.}
\end{equation*}%
On the other hand, let $x\in X_{h}$ and let $w\in X$ be the unique element
such that $x=\theta (w)$. Since $\theta $ is an orbit injection, we conclude
that there exists $k\in \mathbb{Z}^{d}$ such that $\theta (\varphi
^{-k}(w))=\psi ^{-h}(x)$, i.e. $p_{h}^{k}(x)=1$. Since $\theta $, $\psi $
and $\varphi $ are free we conclude that if $k^{\prime }\not=k$ then $%
p_{h}^{k^{\prime }}(x)=0$: indeed, assume there exists $k^{\prime }\in 
\mathbb{Z}^{d}$, $w^{\prime },w^{\prime \prime }\in X$ such that $x=\theta
(w^{\prime })$, $w^{\prime }=\varphi ^{k^{\prime }}(w^{\prime \prime })$ and 
$x=\psi ^{h}\left( \theta \left( w^{\prime \prime }\right) \right) $. Then
since $\theta $ is injective, $w^{\prime }=w$. Since $\psi ^{h}$ and $\theta 
$ are injective, $w^{\prime \prime }=\varphi ^{-k}(w)$. Hence $\varphi
^{-k^{\prime }}(w)=\varphi ^{-k}(w)$ and since $\varphi $ is free we
conclude that $k=k^{\prime }$. Hence the projections $p_{h}^{j}$ for $j\in 
\mathbb{Z}^{d}$ are pairwise orthogonal and thus the sets $X_{h}^{k}$ ($k\in 
\mathbb{Z}^{d}$) are disjoint.

Since $X_{h}$ is compact and $X_{h}$ is the disjoint union of the clopen
subsets $X_{h}^{k}$ ($k\in \mathbb{Z}^{d}$) of $X_{h}$, we also conclude
that the set $\left\{ k\in \mathbb{Z}^{d}:p_{h}^{k}\not=0\right\} $ is
finite. Therefore we can write:%
\begin{equation*}
q_{h}=\sum_{k\in \mathbb{Z}^{d}}p_{h}^{k}\in C(Z)\text{.}
\end{equation*}

Now, for all $z,z^{\prime }\in \mathbb{Z}^{d}$ we have by construction: $%
p_{h}^{z}\omega ^{z}p_{h}^{z^{\prime }}=p_{h}^{z}U_{\psi
}^{h}p_{h}^{z^{\prime }}$ so (noting all the sums are finite):%
\begin{eqnarray*}
pU_{\psi }^{h}p &=&q_{h}U_{\psi }^{h}q_{h}=\sum_{z\in \mathbb{Z}%
^{d}}\sum_{z^{\prime }\in \mathbb{Z}^{d}}p_{h}^{z}U_{\psi
}^{j}p_{h}^{z^{\prime }} \\
&=&\sum_{z\in \mathbb{Z}^{d}}\sum_{z^{\prime }\in \mathbb{Z}%
^{d}}p_{h}^{z}\omega ^{z}p_{h}^{z^{\prime }}\in \func{ran}\alpha \text{.}
\end{eqnarray*}%
Hence the range of $\alpha $ is $p(C(Z)\rtimes _{\psi }\mathbb{Z}^{d})p$ as
claimed.
\end{proof}

\bigskip We now proceed to prove the converse of Proposition (\ref{Morita1})
and thus establish our main theorem for this section. We use the following
notations. Let $G$ be an Abelian discrete group and let $\varphi $ be an
action of $G$ on a compact space $X$. We denote the Pontryagin dual of $G$
by $\widehat{G}$. We define the dual action $\gamma $ of the compact group $%
\widehat{G}$ on $C(X)\rtimes _{\varphi }G$ as the unique action such that
for all $f\in C(X)$,$\chi \in \widehat{G}$ and $g\in G$ we have $\gamma
^{\chi }(f)=f$ and $\gamma ^{\chi }(U^{g})=\chi (g)U^{g}$. In this context,
we define, for all $a\in C(X)\rtimes _{\varphi }G$:%
\begin{equation}
\mathbb{E}(a)=\int_{G}\gamma ^{\chi }(a)d\mu (\chi )  \label{Expectation}
\end{equation}%
where $\mu $ is the Haar probability measure on the compact group $\widehat{G%
}$. It is standard that $\mathbb{E}$ is a conditional expectation from $%
C(X)\rtimes _{\varphi }G$ onto the fixed point of $\gamma $ which is $C(X)$.
We refer to \cite{Zeller-Meier68}, \cite{Pedersen79} or \cite{Williams07}
for the proof of the existence of the strongly continuous action $\gamma $
and its fundamental properties.

We have the following immediate extension of \cite[Lemma 5.1]{Putnam95} from
actions of $\mathbb{Z}$ to actions of $\mathbb{Z}^{d}$:

\begin{lemma}
\label{FundationL1}Let $d\in \mathbb{N}\backslash \left\{ 0\right\} $. Let $%
X $ be an infinite compact space and let $\varphi $ be a minimal free action
of $\mathbb{Z}^{d}$ on $X$ by homeomorphisms. Let $v$ be a unitary in $%
C(X)\rtimes _{\varphi }\mathbb{Z}^{d}$ such that $vC(X)v^{\ast }=C(X)$. Then
there exists orthogonal projections $\left( p_{g}\right) _{g\in \mathbb{Z}%
^{d}}$ in $C(X)$ such that $\sum_{g\in \mathbb{Z}^{d}}p_{g}=1$ while $%
\left\{ g\in \mathbb{Z}^{d}:p_{g}\not=0\right\} $ is finite and:%
\begin{equation*}
v=f\sum_{g\in \mathbb{Z}^{d}}p_{g}U_{\varphi }^{g}
\end{equation*}%
for some $f\in C(X)$.
\end{lemma}

\begin{proof}
Let $G=\mathbb{Z}^{d}$ and let the Pontryagin dual of $G$ be denoted by $%
\widehat{G}$ (i.e. $\widehat{G}=\mathbb{T}^{d}$). For $g\in G$ let $\delta
_{g}$ be the Dirac measure at $g\in G$ (we identify measures with their
density against the counting measure over the countable space $G$). Let $%
x\in X$. Define the following representation $\pi _{x}$ of $C(X)\rtimes
_{\varphi }G$ on $\ell ^{2}(G)$: for $f\in C(X)$ we set $\pi _{x}(f)\delta
_{g}=f\left( \varphi ^{-g}(x)\right) \delta _{g}$ and for $h\in G$ we set $%
\pi _{x}\left( U_{\varphi }^{h}\right) \delta _{g}=\delta _{g+h}$. This
representation is known as the regular representation induced by the measure
Dirac measure at $x$ on $X$ \cite{Pedersen79}. It is routine to check that $%
\pi _{x}(U_{\varphi }^{h})\pi _{x}(f)\pi _{x}\left( U_{\varphi }^{-h}\right)
=\pi _{x}\left( f\circ \varphi ^{-h}\right) $ and thus $\pi _{x}$ extends
uniquely to a *-representation of $C(X)\rtimes _{\varphi }G$. Moreover,
since the action of $G$ is minimal and $X$ is infinite, the crossed-product $%
C(X)\rtimes _{\varphi }G$ is simple and thus $\pi _{x}$ is faithful. In
addition, $\pi _{x}$ is irreducible (using the freedom of the action $%
\varphi $). These facts can be found in \cite{Pedersen79} and are well-known.

Let $g\in G$. Set $p_{g}=\left\vert \mathbb{E}\left( vU^{-g}\right)
\right\vert $ where $\mathbb{E}$ is the conditional expectation on $%
C(X)\rtimes _{\varphi }G$ defined by the dual action $\gamma $ of $\widehat{G%
}$ and Identity (\ref{Expectation}). Let $X_{g}\subseteq X$ be the support
of $p_{g}\in C(X)$. Let $\chi \in \widehat{G}$. We define the unitary $%
u_{\chi }$ on $\ell ^{2}\left( G\right) $ by $u_{\chi }\delta _{h}=\chi
(h)\delta _{h}$ for all $h\in G$. For $f\in C(X)$ we have:%
\begin{eqnarray*}
\pi _{x}\left( \gamma ^{\chi }(f)\right) &=&\pi _{x}(f) \\
&=&u_{\chi }\pi _{x}(f)u_{\chi }^{\ast }
\end{eqnarray*}%
(since $u_{\chi }$ is diagonal in the basis $\left( \delta _{g}\right)
_{g\in G}$ so commutes with the diagonal operators $\pi _{x}(f)$ for all $%
f\in C(X)$). Let $g\in G$. Then:%
\begin{eqnarray*}
u_{\chi }\pi _{x}(U^{g})u_{\chi }^{\ast }\delta _{h} &=&u_{\chi }\pi
_{x}(U^{g})\overline{\chi (h)}\delta _{h} \\
&=&u_{\chi }\pi _{x}(U^{g})\chi (-h)\delta _{h} \\
&=&u_{\chi }\chi (-h+g)\delta _{h+g} \\
&=&\chi (g)\delta _{h+g} \\
&=&\chi (g)\pi _{x}(U^{g})\delta _{h} \\
&=&\pi _{x}\left( \gamma ^{\chi }(U^{g})\right) \delta _{h}\text{.}
\end{eqnarray*}%
Hence since $\pi _{x}$ is a continuous *-morphism, $\pi _{x}\circ \gamma
^{\chi }=\limfunc{Ad}u_{\chi }\circ \pi _{x}$. (Note that since the action
is minimal, $\pi _{x}$ is actually injective on $C(X)$ and together with the
intertwining of the dual actions, this fact ensures an alternative proof
that $\pi _{x}$ is faithful).

Since $\pi _{x}(v)$ stabilizes $\pi _{x}(C(X))$, it also stabilizes $\pi
_{x}\left( C\left( X\right) \right) ^{\prime \prime }$. Now: 
\begin{equation*}
\pi _{x}\left( C\left( X\right) \right) ^{\prime \prime }=\ell ^{\infty
}\left( G\right)
\end{equation*}
where $\ell ^{\infty }\left( G\right) $ is identified with the
multiplication operators on $\ell ^{2}(G)$ or, equivalently, with the
maximal Abelian Von Neumann algebra of diagonal operators in the basis $%
\left\{ \delta _{g}:g\in G\right\} $. Indeed, the inclusion $\ell ^{\infty
}(G)\subseteq \pi _{x}\left( C(X)\right) ^{\prime }$ is easily checked, and
if $T\in \mathcal{B}\left( \ell ^{2}\left( G\right) \right) $ with $%
\left\langle T\delta _{g},\delta _{g^{\prime }}\right\rangle \not=0$ for
some $g\not=g^{\prime }$ then, choosing $f\in C(X)$ so that $f(\alpha
_{g}(x))\not=f(\alpha _{g^{\prime }}(x))$ we conclude that:%
\begin{eqnarray*}
\left\langle T\pi (f)\delta _{g},\delta _{g^{\prime }}\right\rangle
&=&f(\alpha _{g}(x))\left\langle T\delta _{g},\delta _{g^{\prime
}}\right\rangle , \\
\left\langle \pi (f)T\delta _{g},\delta _{g^{\prime }}\right\rangle
&=&f\left( \alpha _{g^{\prime }}(x)\right) \left\langle T\delta _{g},\delta
_{g^{\prime }}\right\rangle ,
\end{eqnarray*}%
and thus $\pi (f)T\not=T\pi (f)$. Hence, $T\in \pi (C(X))^{\prime }$ if and
only if $\varphi \in \ell ^{\infty }(G)$. Hence $\ell ^{\infty }(G)=\pi
_{x}\left( C(X)\right) ^{\prime }$ and thus $\pi _{x}(C(X))^{\prime \prime
}=\ell ^{\infty }(G)^{\prime }=\ell ^{\infty }(G)$.

Thus there exists $\left( \lambda _{h}\right) _{h\in G}$ with $\lambda
_{h}\in \mathbb{T}$ and $\sigma :G\rightarrow G$ a permutation such that for
all $h\in G$:%
\begin{equation*}
\pi _{x}(v)\delta _{h}=\lambda _{h}\delta _{\sigma (h)}\text{.}
\end{equation*}%
(note: if $Q_{h}$ is the projection onto $\mathbb{C}\delta _{h}$ then $\pi
_{x}(v)Q_{h}\pi _{x}(v)^{\ast }\in \ell ^{\infty }(G)$ and is a projection
so $\pi _{x}(v)Q_{h}\pi _{x}(v)^{\ast }=Q_{\sigma (h)}$.)

Then for all $g\in G$ and $h\in G$ we have:%
\begin{eqnarray*}
\pi _{x}\left( \mathbb{E}_{g}(v)\right) \delta _{h} &=&\int_{\widehat{G}}\pi
_{x}\left( \gamma ^{\chi }(\nu U^{-g})\right) \delta _{h}d\mu (\chi ) \\
&=&\int_{\widehat{G}}u_{\chi }\pi _{x}\left( v\right) \pi _{x}\left(
U^{-g}\right) u_{\chi }^{\ast }\delta _{h}d\mu (\chi ) \\
&=&\int_{\widehat{G}}\chi \left( \sigma \left( h-g\right) -h\right) \lambda
_{h-g}\delta _{\sigma (h-g)}d\mu \left( \chi \right) \\
&=&\left\{ 
\begin{array}{c}
\lambda _{g-h}\delta _{h}\text{ if }\sigma (h-g)=h \\ 
0\text{ if }\sigma (h-g)\not=h%
\end{array}%
\right. \text{.}
\end{eqnarray*}%
Hence:%
\begin{equation*}
\pi _{x}\left( p_{g}\right) \delta _{h}=\left\{ 
\begin{array}{c}
\delta _{h}\text{ if }\sigma (h-g)=h\text{,} \\ 
0\text{ if }\sigma (h-g)\not=h\text{.}%
\end{array}%
\right.
\end{equation*}%
This proves that $p_{g}$ is a projection and $p_{g}p_{h}=0$ iff $g\not=h\in
G $. In particular, $X_{g}$ is clopen since $p_{g}$ continuous.

Let $-g_{x}=\sigma ^{-1}(0)$ (note: $\sigma $ depends on $x$ as it is
defined via $\pi _{x}(v)$). Then if $0\in G$ is the neutral element of $G$:%
\begin{equation*}
p_{g_{x}}(x)=\left\langle \pi _{x}\left( p_{g_{x}}\right) \delta _{0},\delta
_{0}\right\rangle =1\text{.}
\end{equation*}%
Since $x$ is arbitrary in $X$, we conclude that $\dbigcup\limits_{g\in
G}X_{g}=X$ (by above equation: if $x\in X$ then $x\in X_{g_{x}}$). Since $%
\left\{ X_{g}:g\in G\right\} $ is an open covering of the compact $X$ there
exists a finite subset $S\subseteq G$ such that $X=\dbigcup\limits_{g\in
S}X_{g}$. Since the sets $X_{g}$ are pairwise disjoint (as $p_{g}p_{h}=0$)\
we conclude that if $g\in S^{c}$ then $X_{g}=\emptyset $. Hence, $p_{g}=0$
if $g\in S^{c}$ and $\oplus _{g\in S}p_{g}=1$.

Last, let $v_{0}=\sum_{g\in S}p_{g}U^{g}$. By construction, $\pi _{x}\left(
vv_{0}^{\ast }\right) =\lambda $. Since $\pi _{x}$ is faithful and unital
and $\lambda $ is a unitary, so is $vv_{0}^{\ast }$ and thus $v_{0}$ is a
unitary. In particular, $\sum_{g\in G}\varphi ^{-g}(p_{g})=v_{0}^{\ast
}v_{0}=1$. Moreover, since $\lambda \in \pi _{x}\left( C(X)\right) ^{\prime
\prime }$ and $C(X)$ is maximal Abelian in $C(X)\rtimes _{\varphi }G$ (as $%
\varphi $ is free), we conclude that there exists $f\in C(X)$ such that $\pi
_{x}(f)=\lambda $. By faithfulness of $\pi _{x}$ we conclude that $v=fv_{0}$
as claimed.
\end{proof}

\bigskip We now prove the main result of this section:

\begin{theorem}
\label{MainChar}Let $\left( X,\varphi ,\mathbb{Z}^{d}\right) $ and $\left(
Z,\psi ,\mathbb{Z}^{d}\right) $ be two free minimal Cantor dynamical
systems. The following are equivalent:

\begin{enumerate}
\item There exists a bounded orbit injection $\theta :X\rightarrow Z$,

\item There exists a *-monomorphism $C(X)\rtimes _{\varphi }\mathbb{Z}^{d}$
into $C(Z)\rtimes _{\psi }\mathbb{Z}^{d}$ such that:

\begin{itemize}
\item $\alpha (C(X))\subseteq C(Z)$,

\item Letting $p=\alpha (1)$, the range of $\alpha $ is $p\left( C(Z)\rtimes
_{\psi }\mathbb{Z}^{d}\right) p$.
\end{itemize}
\end{enumerate}
\end{theorem}

\begin{proof}
Proposition (\ref{Morita1}) establishes that (1) implies (2). We are left to
show that (2) implies (1).

Assume henceforth that we are given a *-monomorphim $\alpha :C(X)\rtimes
_{\varphi }\mathbb{Z}^{d}\rightarrow C(Z)\rtimes _{\psi }\mathbb{Z}^{d}$
with the properties mentioned in (2). Let $p=\alpha (1)\in C(Z)$ ($p$ is a
continuous $\left\{ 0,1\right\} $ valued function on $Z$ whose support is a
clopen subset of $Z$ denoted by $X^{\prime }$). Note that we can write any
operator in $C(Z)\rtimes _{\psi }\mathbb{Z}^{d}$ as a 2 by 2 matrix $\left[ 
\begin{array}{cc}
a & b \\ 
c & d%
\end{array}%
\right] $ such that $a\in p\left( C(Z)\rtimes _{\psi }\mathbb{Z}^{d}\right)
p $, $b\in p\left( C(Z)\rtimes _{\psi }\mathbb{Z}^{d}\right) (1-p)$, $c\in
(1-p)\left( C(Z)\rtimes _{\psi }\mathbb{Z}^{d}\right) p$ and $d\in \left(
1-p\right) \left( C(Z)\rtimes _{\psi }\mathbb{Z}^{d}\right) \left(
1-p\right) $. Now, in this decomposition, for all $g\in \mathbb{Z}^{d}$ we
set $\omega ^{g}=\alpha (U_{\varphi }^{g})$ and $V^{g}=\left[ 
\begin{array}{cc}
\omega ^{g} & 0 \\ 
0 & 1-p%
\end{array}%
\right] $. Fix $g\in \mathbb{Z}^{d}$. By assumption on $\alpha $ we have: 
\begin{equation*}
C(Z)=\left\{ \left[ 
\begin{array}{cc}
\alpha (f) & 0 \\ 
0 & f^{\prime }%
\end{array}%
\right] :f\in C(X),f^{\prime }\in C(Z\backslash X^{\prime })\right\} \text{.}
\end{equation*}%
Fix $g\in \mathbb{Z}^{d}$. We also have by our assumptions on $\alpha $ that
for all $f\in C(X)$:%
\begin{equation*}
\omega ^{g}\alpha (f)\omega ^{-g}=\alpha (U_{\varphi }^{g}fU_{\varphi
}^{-g})=\alpha (f\circ \varphi _{-g})\in pC(Z)p\text{.}
\end{equation*}%
Hence $V^{g}$ stabilizes $C(Z)$ in $C(Z)\rtimes _{\psi }\mathbb{Z}^{d}$. By
Lemma \ref{FundationL1}, we conclude that for all $h\in \mathbb{Z}^{d}$
there exists a projection $p_{h}^{g}\in C(Z)$ such that $\sum_{h\in \mathbb{Z%
}^{d}}p_{h}^{g}=1$ with $\left\{ h\in \mathbb{Z}^{d}:p_{h}^{g}\not=0\right\} 
$ finite and $f_{g}\in C(X^{\prime })$ (valued in $\mathbb{T}$) such that:%
\begin{equation*}
V^{g}=f_{g}\sum_{h\in \mathbb{Z}^{d}}p_{h}^{g}U_{\psi }^{h}\text{.}
\end{equation*}%
By assumption on $\alpha $, the set $X^{\prime }$ is homeomorphic to $X$.
Then let $x\not\in X^{\prime }$. Let $h\in \mathbb{Z}^{d}\backslash \left\{
0\right\} $. Then if $p_{h}^{g}(x)\not=0$ (hence $=1$) then $p_{k}^{g}(x)=0$
for all $k\not=h$ by orthogonality. Let $f^{\prime }\in C(Z)$ supported in $%
Z\backslash X^{\prime }$ and such that $1=f^{\prime }\left( x\right)
\not=f^{\prime }\left( \psi ^{-h}(x)\right) $. Then $V^{g}f^{\prime
}V^{-g}=f^{\prime }$ by construction. Yet:%
\begin{eqnarray*}
&&\left( f_{g}\sum_{k\in \mathbb{Z}^{d}}p_{k}^{g}U_{\psi }^{k}\right)
f^{\prime }\left( f_{g}\sum_{k\in \mathbb{Z}^{d}}p_{k}^{g}U_{\psi
}^{k}\right) ^{\ast } \\
&=&f_{g}\left( \sum_{k\in \mathbb{Z}^{d}}\sum_{k^{\prime }\in \mathbb{Z}%
^{d}}\left( p_{k}^{g}\left( f^{\prime }\circ \psi ^{-k}\right) \left(
U_{\psi }^{k-k^{\prime }}\right) p_{k^{\prime }}^{g}\right) \right)
f_{g}^{\ast } \\
&=&f_{g}\left( \sum_{k\in \mathbb{Z}^{d}}\sum_{k^{\prime }\in \mathbb{Z}%
^{d}}\left( f^{\prime }\circ \psi ^{-k}\right) \left( p_{k}^{g}p_{k^{\prime
}}^{g}\circ \psi ^{k^{\prime }-k}\right) U_{\psi }^{k-k^{\prime }}\right)
f_{g}^{\ast }\text{,}
\end{eqnarray*}%
and thus, by orthogonality, as in Identity (\ref{Morita_id1}) in the proof
of Proposition (\ref{Morita1}), we see that:%
\begin{equation*}
\left( f_{g}\sum_{k\in \mathbb{Z}^{d}}p_{k}^{g}U_{\psi }^{k}\right)
f^{\prime }\left( f_{g}\sum_{k\in \mathbb{Z}^{d}}p_{k}^{g}U_{\psi
}^{k}\right) ^{\ast }=\sum_{k\in \mathbb{Z}^{d}}p_{k}^{g}\cdot f^{\prime
}\circ \psi ^{-k}\text{.}
\end{equation*}%
Hence:%
\begin{eqnarray*}
\left( \sum_{k\in \mathbb{Z}^{d}}p_{k}^{g}U_{\psi }^{k}\right) f^{\prime
}\left( \sum_{k\in \mathbb{Z}^{d}}p_{k}^{g}U_{\psi }^{k}\right) ^{\ast }(x)
&=&p_{h}^{g}(x)\cdot f^{\prime }\circ \psi ^{-h}(x)\cdot p_{h}^{g}(x) \\
&=&f^{\prime }\circ \psi ^{-h}(x)\not=f^{\prime }(x)\text{.}
\end{eqnarray*}

This is a contradiction.\ Hence $p_{h}^{g}(x)=0$ and $p_{h}^{g}\in
C(X^{\prime })$ for all $h\not=0$ and $g\in G$. Clearly $1-p$ is a
subprojection of $p_{0}^{g}$ for all $g\in G$. It will be convenient to
change our notation in the sequel, and denote the projection $p_{0}^{g}p$ by 
$p_{0}^{g}$ for all $g\in \mathbb{Z}^{d}$. With this new notation, $%
\sum_{h\in \mathbb{Z}^{d}}p_{h}^{g}=p$.

Let $\theta :X\rightarrow X^{\prime }$ be the homeomorphism defined by $%
\alpha (f)(x)=f\circ \theta ^{-1}(x)$ for all $x\in X^{\prime }$ and $f\in
C(X)$. We claim that $\theta $, identified as an injection $X\rightarrow Z$,
is a bounded orbit injection.

Let $x\in X$ and $g\in \mathbb{Z}^{d}$. Let $X_{h}^{g}$ be the support of $%
p_{h}^{g}$ for all $h\in \mathbb{Z}^{d}$ and note that $X_{h}^{g}$ is a
clopen subset of $X^{\prime }$. Since $\sum_{h\in \mathbb{Z}^{d}}p_{h}^{g}=p$
and these projections are orthogonal, there exists a unique $h=h\left(
x,g\right) \in \mathbb{Z}^{d}$ such that $\theta (\varphi ^{g}(x))\in
X_{h}^{g}$. Now, since there are only finitely many nonzero projections $%
\left\{ p_{h}^{g}:h\in \mathbb{Z}^{d}\right\} $, the map $x\mapsto h\left(
x,g\right) $ is bounded on $X$.

Let $x,y\in X$ and set $z=\theta (y)$. Then $\varphi ^{g}(x)=y$ implies $%
z\in X_{h(x,g)}^{g}$. Let $h=h(x,g)$. Now by construction, $z\in
X_{h(x,g)}^{g}$ if and only if $p_{k}^{g}(z)=0$ for all $k\in \mathbb{Z}%
^{d}\backslash \{h\}$. Using the same computation technique as before, we
have:%
\begin{eqnarray}
\alpha (f)\circ \psi ^{-h(x,g)}(z) &=&U_{\psi }^{h(x,g)}\alpha (f)U_{\psi
}^{-h(x,g)}(z)  \notag \\
&=&\left( \sum_{k\in \mathbb{Z}^{d}}p_{k}^{g}U_{\psi }^{k}\right) \cdot
\alpha (f)\cdot \left( \sum_{k\in \mathbb{Z}^{d}}p_{k}^{g}U_{\psi
}^{k}\right) ^{\ast }(z)  \notag \\
&=&\left( V^{g}\alpha (f)V^{-g}\right) (z)  \notag \\
&=&\alpha (U_{\varphi }^{g}fU_{\varphi }^{-g})(z)  \notag \\
&=&\alpha \left( f\circ \varphi _{-g}\right) (z)\text{.}  \label{morita2_eq1}
\end{eqnarray}

Hence $\varphi ^{g}(x)=y$ implies by definition of $\alpha $ and Equality (%
\ref{morita2_eq1}):%
\begin{equation*}
f\circ \varphi ^{-g}(y)=f\circ \theta ^{-1}\circ \psi ^{-h(x,g)}\circ \theta
(y)
\end{equation*}
for all $f\in C(X)$, or equivalently as $C(X)$ separates the points of $X$:%
\begin{equation*}
x=\varphi ^{-g}(y)=\theta ^{-1}\circ \psi ^{h(x,g)}\circ \theta (y)\text{,}
\end{equation*}%
or equivalently: $\psi ^{-h(x,g)}\left( \theta (x)\right) =\theta (y)$.

Let us now assume that $\psi ^{k}\left( x\right) =y$ for some $x,y\in
X^{\prime }$ and $k\in \mathbb{Z}^{d}$. We shall use the following sequence
of claims to establish the existence of $n(x,g)\in \mathbb{Z}^{d}$ such that 
$\varphi ^{n(x,g)}(\theta ^{-1}(x))=\theta ^{-1}(y)$.

Let $A=C(Z)\rtimes _{\psi }\mathbb{Z}^{d}$. The dual action $\gamma $ of $%
\mathbb{T}^{d}$ on $A$ defines spectral subspaces by setting:%
\begin{equation*}
\forall g\in \mathbb{Z}^{d}\ \ \ A_{g}=\left\{ a\in A:\forall \omega \in 
\mathbb{T}^{d}\ \ \ \gamma ^{\omega }(a)=\omega ^{g}a\right\}
\end{equation*}%
where, if we write $\omega =\left( \omega _{1},\ldots ,\omega _{d}\right) $
and $g=\left( g_{1},\ldots ,g_{d}\right) $ then $\omega ^{g}=\left( \omega
_{1}^{g_{1}},\ldots ,\omega _{d}^{g_{d}}\right) $. For $a,b\in A$ we set:%
\begin{equation*}
\left\langle a,b\right\rangle =\mathbb{E}\left( b^{\ast }a\right)
\end{equation*}%
and we thus defined an $C(Z)$-valued inner product. We check easily that $%
A_{k}$ and $A_{m}$ are orthogonal for $k\not=m$ for this inner product.

We now prove:

\begin{claim}
\label{ClaimPUP}There exists $x\in X^{\prime }$ and $k\in \mathbb{Z}^{d}$
with $\psi ^{k}(x)\in X^{\prime }$ if and only if $pU_{\psi }^{k}p\not=0$.
\end{claim}

Note that:%
\begin{eqnarray}
\left( pU_{\psi }^{k}p\right) ^{\ast }\left( pU_{\psi }^{k}p\right)
&=&p\left( p\circ \psi ^{k}\right) p  \notag \\
&=&p\left( p\circ \psi ^{k}\right) \text{.}  \label{ClaimPUP_eq1}
\end{eqnarray}%
Hence $pU_{\psi }^{k}p\not=0$ if and only if $\left( pU_{\psi }^{k}p\right)
\left( pU_{\psi }^{k}p\right) ^{\ast }\not=0$ which is equivalent by (\ref%
{ClaimPUP_eq1})\ to the existence of $x\in X^{\prime }$ such that $\psi
^{k}(x)\in X^{\prime }$.

\begin{claim}
\label{ClaimPUPAK}For any $k\in \mathbb{Z}^{d}$ we have $pU_{\psi }^{k}p\in
A_{k}$.
\end{claim}

Let $\chi \in \mathbb{T}^{d}$. Then, note that $p=\alpha (1)\in C(Z)$ by
assumption so $\gamma ^{\chi }(p)=p$ for all $\chi \in \mathbb{T}^{d}$ hence:%
\begin{equation*}
\gamma ^{\chi }(pU_{\psi }^{k}p)=\gamma ^{\chi }(p)\gamma ^{\chi }(U_{\psi
}^{k})\gamma ^{\chi }(p)=p\gamma ^{\chi }(U_{\psi }^{k})p=\chi ^{k}pU_{\psi
}^{k}p
\end{equation*}%
as claimed.

\begin{claim}
\label{ClaimNullPg}Suppose that there exists $k\in \mathbb{Z}^{d}$ such that
for all $g\in \mathbb{Z}^{d}$ and $x\in X^{\prime }$ we have $%
p_{k}^{g}\left( x\right) =0$. Then $\alpha (C(X)\rtimes _{\varphi }\mathbb{Z}%
^{d})$ is orthogonal to $A_{k}$ for $\left\langle .,.\right\rangle $.
\end{claim}

The C*-algebra $\alpha (C(X)\rtimes _{\varphi }\mathbb{Z}^{d})$ is generated
by $\alpha (C(X))\subseteq C(Z)$ and $\omega ^{g}$ ($g\in \mathbb{Z}^{d}$).
By Lemma\ \ref{FundationL1}, for all $g\in \mathbb{Z}^{d}$ we have $\omega
^{g}=f_{g}\sum_{h\in \mathbb{Z}^{d}}p_{h}^{g}U_{\psi }^{h}$ after replacing $%
p_{0}^{g}$ with $p_{0}^{g}p$ and with $f_{g}\in C(X^{\prime })$. Thus with
our assumption, $\omega ^{g}\perp A_{k}$. Hence $\alpha (C(X)\rtimes
_{\varphi }\mathbb{Z}^{d})$ is orthogonal to $A_{k}$.

\begin{claim}
\label{ClaimNullCondPg}Suppose that there exists $x\in X^{\prime }$ and $%
k\in \mathbb{Z}^{d}$ such that $y:=\psi ^{k}(x)\in X^{\prime }$. Assume
moreover that there is no $g\in \mathbb{Z}^{d}$ such that $\theta
^{-1}(y)=\varphi ^{g}(\theta ^{-1}(x))$. Then for all $g\in \mathbb{Z}^{d}$
and $z\in X^{\prime }$ we have $p_{k}^{g}(z)=0$.
\end{claim}

Let $g,j\in \mathbb{Z}^{d}$. By definition, if $p_{k}^{g}\left( \theta \circ
\varphi ^{j}\circ \theta ^{-1}(x)\right) \not=0$ then 
\begin{equation*}
y=\theta (\varphi ^{g}(\theta ^{-1}(\theta \circ \varphi ^{j}\circ \theta
^{-1}(x))))=\theta \left( \varphi ^{g+j}\left( \theta ^{-1}(x)\right)
\right) \text{.}
\end{equation*}%
So by assumption $p_{k}^{g}\left( \theta \circ \varphi ^{j}\circ \theta
^{-1}(x)\right) =0$ for all $g,j\in \mathbb{Z}^{d}$. Now:%
\begin{equation*}
\overline{\left\{ \varphi ^{j}\left( \theta ^{-1}(x)\right) :j\in \mathbb{Z}%
^{d}\right\} }=X
\end{equation*}%
since $\varphi $ is minimal.

Since $\theta $ and $p_{k}^{g}$ are continuous and $p_{k}^{g}\circ \theta $
is null on $\left\{ \varphi ^{j}\left( \theta ^{-1}(x)\right) :j\in \mathbb{Z%
}^{d}\right\} $, and since $\theta (X)=X^{\prime }$, we conclude that $%
p_{k}^{g}(X^{\prime })=\left\{ 0\right\} $ for all $g\in \mathbb{Z}^{d}$.
Since $p_{k}^{g}$ is supported on $X^{\prime }$ by construction, we conclude
that $p_{k}^{g}$ is null on $X^{\prime }$ for all $g\in \mathbb{Z}^{d}$ as
claimed.

\begin{claim}
If there exists $x\in X^{\prime }$ such that $\psi ^{k}(x)=y\in X^{\prime }$
then there exists $g\in \mathbb{Z}^{d}$ such that $\varphi ^{g}(\theta
^{-1}(x))=\theta ^{-1}(y)$.
\end{claim}

By Claim (\ref{ClaimPUP}), we conclude $pU_{\psi }^{k}p\not=0$. Since $%
\alpha $ is onto $pAp$, we conclude $pU_{\psi }^{k}p\in \func{ran}\alpha $.
Yet $pU_{\psi }^{k}p\in A_{k}$ by Claim (\ref{ClaimPUPAK}). If we assume
that there exists no $g\in \mathbb{Z}^{d}$ such that $\varphi ^{g}(x)=y$
then by Claim (\ref{ClaimNullCondPg}) $p_{k}^{g}=0$ for all $g\in \mathbb{Z}%
^{d}$. By Claim (\ref{ClaimNullPg}), this would imply that $\func{ran}\alpha 
$ is orthogonal to $A_{k}$. We have reached a contradiction. Hence there
exists $g\in \mathbb{Z}^{d}$ such that $\theta \circ \varphi ^{g}\circ
\theta ^{-1}(x)=y$.

Hence we conclude that $\theta $ is a bounded orbit injection as desired.
\end{proof}

\bigskip We note that our result proves, among other things, that the corner 
$p\left( C(Z)\rtimes _{\psi }\mathbb{Z}^{d}\right) p$ with the notation of
Theorem (\ref{MainChar}) is in fact a C*-crossed-product itself, a fact
which follows in our proof from the special nature of bounded orbit
injections.

\bigskip The ordered $K_{0}$ group of C*-algebras is an invariant for
Rieffel-Morita equivalence, and Theorem (\ref{MainChar}) implies that two
minimal free Cantor systems are bounded orbit injection equivalence must
have Rieffel-Morita equivalent C*-crossed-products. Hence the ordered $K_{0}$
group of C*-crossed-products of minimal free Cantor systems is an invariant
of bounded orbit injection equivalence. However the computation of the $%
K_{0} $ group of a C*-crossed-product of a minimal free actions of $\mathbb{Z%
}^{d}$ on a Cantor set is a delicate matter for $d>1$, as shown for instance
in \cite{Phillips05} -- unless in the case $d=1$ where the
Pimsner-Voiculescu six-term exact sequence suffices \cite{Blackadar98}. As
we shall see in the next section, one can however consider an alternative
ordered group as an invariant for orbit injection equivalence, hence for the
equivalence described in Theorem (\ref{MainChar}) between
C*-crossed-products of minimal free Cantor actions.

\section{Orbit Equivalence from Orbit Injection Equivalence}

If $\left( X,\varphi ,\mathbb{Z}^{d}\right) $ and $\left( Y,\psi ,\mathbb{Z}%
^{d}\right) $ are bounded orbit injection equivalent, there is a third
system $\left( Z,\alpha ,\mathbb{Z}^{d}\right) $ and bounded orbit
injections from both $\left( X,\varphi ,\mathbb{Z}^{d}\right) $ and $\left(
Y,\psi ,\mathbb{Z}^{d}\right) $ into $\left( Z,\alpha ,\mathbb{Z}^{d}\right) 
$. In this section, we focus our attention on the problem of omitting the
system $\left( Z,\alpha ,\mathbb{Z}^{d}\right) $. That is, we address the
question of when the existence of a bounded orbit injection equivalence be
strengthened to the existence of a bounded orbit injection and/or an orbit
equivalence.

A particularly useful invariant of bounded orbit equivalence for our purpose
is an ordered group defined by:

\begin{eqnarray*}
G\left( \varphi \right) &=&C\left( X,\mathbb{Z}\right) /\langle f-f\varphi
^{v}:v\in \mathbb{Z}^{d}\rangle \\
G\left( \varphi \right) _{+} &=&\left\{ \left[ f\right] :f\left( x\right)
\geq 0\text{ for all }x\in X\right\} \text{.}
\end{eqnarray*}%
This ordered group has been studied for minimal $\mathbb{Z}^{d}$ systems,
e.g. in \cite{Forrest, Matui08}, and is referred to there as the dynamical
cohomology group. For $d=1$, the triple formed by this ordered group along
with a distinguished order unit forms a complete invariant for strong orbit
equivalence. What, if any, analogous theorem there is for $d>1$ is an open
problem.

In \cite{Lightwood07}, Theorem (2.1)\ proves that if the suspension spaces
of two minimal free Cantor systems $\left( X,\varphi ,\mathbb{Z}^{d}\right) $
and $\left( Y,\psi ,\mathbb{Z}^{d}\right) $ are homeomorphic then $\left(
X,\varphi ,\mathbb{Z}^{d}\right) $ and $\left( Y,\psi ,\mathbb{Z}^{d}\right) 
$ are bounded-orbit injection equivalent. The converse is proven for $d\in
\left\{ 1,2\right\} $ in \cite{Lightwood07}. Thus, one can apply some
results proven in this section if we know that two minimal free Cantor
systems $\left( X,\varphi ,\mathbb{Z}^{d}\right) $ and $\left( Y,\psi ,%
\mathbb{Z}^{d}\right) $ are suspension equivalent. We work here with the
simpler relation of bounded orbit injection equivalence.

\subsection{Voronoi-Rohlin Partitions}

We create Voronoi-Rohlin partitions here, using techniques similar to those
used in \cite{Forrest, Lightwood07, Putnam08, Putnam09}, and analogous to
the way that Rohlin tower partitions are used in \cite[Theorem 2.6]{Putnam95}
for $\mathbb{Z}$-actions. More sophisticated results are in \cite{Putnam09},
we include the necessary arguments here for completeness. We use Voronoi
tilings to tile $\mathbb{Z}^{d}$ for each point $x$, thus partitioning the
points in each $\varphi $-orbit into equivalence classes. The idea is that
for two points $x,y$ in the same equivalence class, we can determine a
vector $v\left( x\right) \in \mathbb{Z}^{d}$ such that $\varphi ^{v\left(
x\right) }\left( x\right) =y$. Since we define $v\left( x\right) $ in a
locally constant\ way, we are able to define a continuous cocycle $v\left(
x\right) $ to serve our purposes.

We begin with some basics of Voronoi tilings associated with an $M$-regular
set $R\subset \mathbb{R}^{d}$. Below, we denote the Euclidean metric in $%
\mathbb{R}^{d}$ by $d$ , and the Euclidean norm by $\left\Vert .\right\Vert $%
.

\begin{definition}
Let $R\subset \mathbb{Z}^{d}$ and $M>0$. We say that $R$ is $M$-regular if

\begin{enumerate}
\item $R$ is $M$-\emph{separated,} that is, for any $v\in R$, if $w\in R$
and $w\neq v$ then $d\left( v,w\right) >M$.

\item $R$ is $2M$-\emph{syndetic}, that is, for any $v\in \mathbb{R}^{d}$,
there is a $w\in R$ such that $d\left( v,w\right) <2M$
\end{enumerate}
\end{definition}

\bigskip

Specifically, suppose $R\subset \mathbb{Z}^{d}$ is an $M$-regular set, and $%
M>2$. For each $p\in \mathbb{R}^{d}$, let $v\left( p\right) $ be the
(typically singleton) set of points $w\in R$ which achieve $%
\min\limits_{w\in R}d\left( p,w\right) =d\left( p,v\left( p\right) \right) $%
. Note that $\min\limits_{w\in R}d\left( p,w\right) $ is well-defined and
uniformly bounded by $2M$ because the set $R$ is $2M$-syndetic. For $w\in R$%
, the tile containing $w$ is the set $T\left( w\right) =\left\{ p\in \mathbb{%
R}^{d}:w\in v\left( p\right) \right\} $. The covering of $\mathbb{R}^{d}$ by
the tiles $\left\{ T\left( w\right) :w\in R\right\} $ is what we refer to as
the \emph{Voronoi tiling }$\tau \left( R\right) $\emph{. }

We note that with this setup, each tile $T\left( w\right) $ is a convex,
compact subset of $\mathbb{R}^{d}$ which is the intersection of a finite
number of closed half-spaces. Because $R\subset \mathbb{Z}^{d}$ and $R$ is $%
2M$-syndetic, there are only finitely many tiles up to translation. That is,
there are only finitely many different sets $P_{1},P_{2},\ldots ,P_{k}$ of
the form $T\left( w\right) -w$ where $w\in R$, The sets $P_{i}$ are referred
to as the \emph{prototiles} of the Voronoi tiling $\tau \left( R\right) $.

The following will serve as an important preliminary result because it will
later be used to bound the number of vectors that are near the boundary of a
tile.

\begin{lemma}
\label{adjacent}Let $d\geq 1$. Then there is a $b>0$ (depending only on $d$)
such that for any $M>2$ and any $M$-regular set $R\subset \mathbb{Z}^{d}$ if 
$\tau \left( R\right) $ is the Voronoi tiling of $\mathbb{R}^{d}$ associated
with $R$ then any tile in $\tau $ intersects at most $b$ other tiles.
\end{lemma}

\begin{proof}
Let $B\left( x,r\right) \subset \mathbb{R}^{d}$ denote the ball in $\mathbb{R%
}^{d}$ centered at $x$ and with radius $r$. Let $b$ be the maximum
cardinality of a set $\left\{ y_{1},y_{2},\ldots ,y_{n}\right\} $ such that $%
y_{i}\in B\left( 0,2\right) $ for all $i$, and $B\left( y_{i},1/2\right)
\cap B\left( y_{j},1/2\right) =\emptyset $ for $i\neq j$.

Now suppose $M>2$ and that $R\subset \mathbb{Z}^{d}$ is an $M$-regular set.
Let $T_{0}\ $be a tile in $\tau \left( R\right) $, and let $a\left(
T_{0}\right) =\left\{ T_{1},T_{2},\ldots ,T_{n}\right\} $ be the set of
tiles in $\tau \left( R\right) $ which intersect, but are not equal to $%
T_{0} $. Let $x_{i}\in R$ denote the center of $T_{i}$ for $0\leq i\leq n$.
A tile $T_{i}$ intersects $T_{j}$ if there is a point $p\in \mathbb{R}^{d}$
such that $d\left( p,x_{i}\right) =d\left( p,x_{j}\right) =\min_{v\in
R}d\left( p,v\right) $. Because $R$ is $2M$-syndetic, $d\left(
x_{0},x_{i}\right) <4M$ for all $1\leq i\leq n$. Because $R$ is $M$%
-separated, $d\left( x_{i},x_{j}\right) >M$ meaning that $B\left(
x_{i},M/2\right) \cap B\left( x_{j},M/2\right) =\emptyset $ for any $1\leq
i,j\leq n$.

Now let $y_{i}=\frac{1}{M}\left( x_{i}-x_{0}\right) $. From the above it
follows that the points $y_{i}$ are in $B\left( 0,2\right) $ for $0\leq
i\leq n$ and $B\left( y_{i},1/2\right) \cap B\left( y_{j},1/2\right)
=\emptyset $ for $0\leq i<j\leq n$. Therefore $n\leq b$.
\end{proof}

\bigskip

To construct partitions associated with Voronoi tilings, suppose $\left(
X,\varphi ,\mathbb{Z}^{d}\right) $ is a minimal Cantor system. Let $C\subset
X$ be clopen, and for each $x\in X$ consider the set of return times of $x$
to $C$: $R_{C}\left( x\right) =\left\{ v\in \mathbb{Z}^{d}:\varphi
^{v}\left( x\right) \in C\right\} \subset \mathbb{Z}^{d}$.

\begin{definition}
Let $C\subset X$ be clopen. We say that $C$ is $M$-regular if for all $x\in
X $, the set $R_{C}\left( x\right) $ is $M$-regular.
\end{definition}

The following propositions establish the existence of $M$-regular clopen
sets $C$ with various properties.

\begin{proposition}
\label{sep}Let $\left( X,\varphi ,\mathbb{Z}^{d}\right) $ be a minimal
Cantor system, and let $M>2$. There is an $\varepsilon >0$ such that if $%
C\subset X$ is clopen with $diam(C)<\varepsilon $ then $C$ is $M$-separated.
\end{proposition}

\begin{proof}
Suppose not. Then there is an $M>2$ such that for every $n\geq 1$, there is
a clopen set $C_{n}\subset X$, a point $x_{n}\in C_{n}$ and a vector $%
v_{n}\in \mathbb{Z}^{d}$ such that $diam\left( C_{n}\right) <\frac{1}{n}$, $%
\varphi ^{v_{n}}\left( x_{n}\right) \in C_{n}$ and $0<\left\Vert
v_{n}\right\Vert <M$. There is a subsequence $\left\{ x_{n_{k}}\right\} $
such that all vectors $v_{n_{k}}$ are equal to a vector $v$ with $%
0<\left\Vert v\right\Vert <M$, and the sequence $\left\{ x_{n_{k}}\right\} $
converges to a point $x_{0}\in X$. Since $\varphi ^{v}$ is continuous, $%
\lim_{k\rightarrow \infty }\varphi ^{v}\left( x_{n_{k}}\right) =\varphi
^{v}\left( x\right) $. Since $d\left( \varphi ^{v}\left( x_{n_{k}}\right)
,x_{n_{k}}\right) <\frac{1}{n_{k}}$, $\lim_{k\rightarrow \infty }\varphi
^{v}\left( x_{n_{k}}\right) =x$, implying that $\varphi ^{v}\left( x\right)
=x$, which is a contradiction to the freeness of the action.
\end{proof}

\begin{proposition}
\label{synd}Let $\left( X,\varphi ,\mathbb{Z}^{d}\right) $ be a minimal
Cantor system. Let $C$ be any clopen set. There is an $r>0$ such that $C$ is 
$r$-syndetic.
\end{proposition}

\begin{proof}
Since $\varphi $ is minimal, $\cup _{v\in \mathbb{Z}^{d}}\varphi ^{v}C$ is
an open cover of $X$. Because $X$ is compact, there is a finite subcover.
The result follows.
\end{proof}

\begin{proposition}
\label{reg}Let $\left( X,\varphi ,\mathbb{Z}^{d}\right) $ be a minimal
Cantor system. Let $C\subset X$ be any nonempty clopen set and $x_{0}\in C$.
Then there is an $M_{0}$ such that if $M>M_{0}$ then there exists a clopen
set $D\subset C$ such that $x_{0}\in D$ and $D$ is $M$-regular.
\end{proposition}

\begin{proof}
By the Proposition \ref{synd}, there is an $r>0$ such that $C$ is $r$%
-syndetic. Let $M_{0}=r$.

Assume that $M>M_{0}$. By Proposition \ref{sep}, we may partition $C$ into
finitely many clopen sets of the form $C_{i}$, $1\leq i\leq I$ where each $%
C_{i}$ is $M$-separated. Without loss of generality, assume $x_{0}\in C_{1}$.

Set $D_{1}=C_{1}$ and for $1<n\leq I$, recursively define: 
\begin{equation*}
D_{n}=D_{n-1}\cup \left( C_{n}\setminus \cup _{\left\Vert v\right\Vert \leq
M}\varphi ^{v}D_{n-1}\right) \text{.}
\end{equation*}

Set $D=D_{I}$. Now suppose that $x,y\in D$ where $y=\varphi ^{w}\left(
x\right) $ with $w\neq 0$.

\textbf{Case 1: }Suppose $x,y\in C_{i}$ for some $i$. Then $\left\Vert
w\right\Vert >M$ since each $C_{i}$ is $M$-separated.

\textbf{Case 2: }Suppose $x\in C_{i}$, $y\in C_{j}$ for $j>i$. Then since $%
x\in D$, we have that $x\in D_{i}$ and therefore $x\in D_{j-1}$. Since $y\in
D$ and $y\in C_{j}$ we have $y\not\in \cup _{\left\Vert v\right\Vert \leq
M}\varphi ^{v}D_{j-1}$. Therefore, $\left\Vert w\right\Vert >M$ and $D$ is $%
M $-separated.

Fix any $x\in X$. To show that $D$ is $2M$-syndetic, it suffices to show
that there is a vector $w$ with $\left\Vert w\right\Vert <2M$ such that $%
\varphi ^{w}\left( x\right) \in D$.

Because $C$ is $M_{0}$-syndetic, we know there is a vector $w$ with $%
\left\Vert w\right\Vert <M_{0}<M$ such that $\varphi ^{w}\left( x\right) \in
C$. Therefore, $\varphi ^{w}\left( x\right) \in C_{i}$ for some $i$. If $%
\varphi ^{w}\left( x\right) \in D_{i}\subset D$, we are done. Otherwise, $%
\varphi ^{w}\left( x\right) \not\in D_{i}$ which implies $\varphi ^{w}\left(
x\right) \in \cup _{\left\Vert v\right\Vert \leq M}\varphi ^{v}D_{i-1}$. But
then there is a $v$ with $\left\Vert v\right\Vert \leq M$ such that $\varphi
^{w}\left( x\right) \in \varphi ^{v}D_{i-1}$ which implies $\varphi
^{w-v}\left( x\right) \in D_{i-1}\subset D$ where $d\left( w,v\right)
<M+M_{0}<2M$.
\end{proof}

Now given $M>2$, we can create a tiling for each point $x\in X$ in the
following way. Fix an $M$-regular clopen set $C$ and for $x\in X$ let $\tau
\left( x,C\right) =\tau \left( R_{C}\left( x\right) \right) $. There will
only be finitely many different prototile sets $P_{1},P_{2},\ldots ,P_{K}$
of the form $T\left( w\right) -w$ where $w\in R_{C}\left( x\right) $ even as
we vary $x$ over the entire space $X$. For $a\in C$, let $P\left( a\right) $
be the prototile containing the origin in $\tau \left( x,C\right) $. Fix a
prototile $P_{k}$ and set $C_{k}=\{a\in C:P\left( a\right) =P_{k}\}$. Then $%
C_{k}=\{a\in C:P\left( a\right) =P_{k}\}$ is clopen because $P\left(
a\right) $ only depends upon the set $B\left( 0,4M\right) \cap R_{C}\left(
a\right) $.

The above gives us a procedure for producing a certain clopen cover $%
\mathcal{A}=\left\{ \varphi ^{w}C_{k}:w\in P_{k},1\leq k\leq K\right\} $ of $%
X$ where $\cup _{k=1}^{K}C_{k}=C$. We note that given the above along with
any finite clopen partition $\mathcal{P}$ of $C$, we will frequently take a
partition 
\begin{equation*}
\mathcal{Q}=\left\{ C_{k}:1\leq k\leq K\right\} \vee \mathcal{P}
\end{equation*}%
of $C$, and consider the cover:%
\begin{equation*}
\mathcal{B}=\left\{ \varphi ^{w}D_{k}^{j}:D_{k}^{j}\subset
C_{k},D_{k}^{j}\in \mathcal{Q},w\in P_{k},1\leq k\leq K\right\}
\end{equation*}%
We will refer to this procedure as "refining $\left\{ C_{k}\right\} $ if
necessary". We note that after such a refinement, one may have $P\left(
x\right) =P\left( y\right) $ for $x,y\in C$ in different partition elements.
Points $x,y\in C$ in the same partition element will always have $P\left(
x\right) =P\left( y\right) $.

What we really need from each tiling $\tau \left( x,C\right) $ is a
partition of $\mathbb{Z}^{d}$. To this end, note that by refining $\left\{
C_{k}\right\} $ if necessary, we may assume that if $x,\varphi ^{v}\left(
x\right) \in C_{k}$ for some $k$, then $\left\Vert v\right\Vert >4M$, which
insures that if $x$ and $\varphi ^{v}\left( x\right) $ are in the same
partition element, the tile centered at the origin in $\tau \left(
x,C\right) $ does not meet the tile centered at $v$. Now fix $x\in C$ and a
tiling $\tau \left( x,X\right) $. For each $w\in R_{C}\left( x\right) $,
consider the set $Z\left( w\right) =T\left( w\right) \cap \mathbb{Z}^{d}$,
the tile centered at $w$ intersected with $\mathbb{Z}^{d}$. The sets $%
\left\{ Z\left( w\right) :w\in R_{C}\left( x\right) \right\} $ need not be
pairwise disjoint. Thus for each $w\in R_{C}\left( x\right) $, we let $%
Z^{\prime }\left( w\right) $ be the set of all elements of $Z\left( w\right) 
$ which do not have the property that $v\in Z\left( u\right) $ with $\varphi
^{w}\left( x\right) \in C_{k}$, $\varphi ^{u}\left( x\right) \in C_{l}$ and $%
l<k$. As a result for each $x\in X$, we obtain a partition $\mathcal{Z}%
\left( x\right) =\left\{ Z^{\prime }\left( w\right) :w\in R_{C}\left(
x\right) \right\} $ of $\mathbb{Z}^{d}$. Again there are only finitely many
different $\mathbb{Z}^{d}$-prototile sets $Z_{1},Z_{2},\ldots ,Z_{K}$ of the
form $Z^{\prime }\left( w\right) -w$ where $w\in R_{C}\left( x\right) $ for $%
x\in X$, and for fixed $k$, the set of points $x$ such that $Z^{\prime
}\left( 0\right) \in \mathcal{Z}\left( x\right) $ is equal to $Z_{k}$ is
clopen.

Mostly inherent in the above discussion is the proof of the following
theorem. Below, $A\vartriangle B$ denotes the symmetric difference of two
sets $A$ and $B$.

\begin{theorem}
\label{voronoi}Let $\left( X,\varphi ,\mathbb{Z}^{d}\right) $ be a minimal
Cantor system. Let $x_{0}\in C\subset X$ where $C$ is clopen. Then there are
integers $L\geq 1$ and $M_{0}>2$, such that given any $M>M_{0}$ there is a
clopen partition $\mathcal{P}=\left\{ \varphi ^{w}C_{k}:w\in Z_{k},1\leq
k\leq K\right\} $ of $X$ with the following properties.

\begin{enumerate}
\item $x_{0}\in \cup _{k=1}^{K}C_{k}\subset C$,

\item for all $k$, $Z_{k}$ contains all vectors $v\in \mathbb{Z}^{d}$ with $%
\left\Vert v\right\Vert \leq M/2$,

\item for all $k$, $Z_{k}$ contains no vectors $v\in \mathbb{Z}^{d}$ with $%
\left\Vert v\right\Vert \geq 2M$,

\item for any $v\in \mathbb{Z}^{d}$, $\left\vert Z_{k}\vartriangle \left(
Z_{k}-v\right) \right\vert <L\left\Vert v\right\Vert M^{d-1}$.
\end{enumerate}
\end{theorem}

\begin{proof}
Given $C$ as in the hypothesis, by Proposition \ref{reg}, there is an $M_{0}$
such that for $M>M_{0}$ there is an $M$-regular clopen set $D$ containing $%
x_{0}$. Let $P_{1},P_{2},\ldots ,P_{K}$ be the prototiles appearing in $\tau
\left( R_{D}\left( x\right) \right) $ for $x\in X$, and $C_{k}=\{x\in
C:P\left( x\right) =P_{k}\}$. Via the procedure described after Proposition %
\ref{reg}, we obtain clopen partition $\mathcal{P}=\left\{ \varphi
^{w}C_{k}:w\in Z_{k},1\leq k\leq K\right\} $ of $X$. Note that any
difference between $P_{k}\cap \mathbb{Z}^{d}$ and $Z_{k}$ takes place only
if we have an intersection of two sets $\varphi ^{w}C_{k}$ and $\varphi
^{v}C_{l}$ in the original cover and in this case $\left\Vert w\right\Vert
=\left\Vert v\right\Vert >M/2$. Thus the intersection of the original $P_{k}$
with the ball around the origin of radius $M/2$ is not affected, and it
follows that $Z_{k}$ satisfy properties 2 and 3. By Lemmas \ref{adjacent},
the collection of vectors in $Z_{k}\vartriangle \left( Z_{k}-v\right) $ is
the union of at most $b$ sets, each of which is a collection of $\mathbb{Z}%
^{d}$-vectors within distance $\left\Vert v\right\Vert $ of a subset of a $%
\left( d-1\right) $-dimensional hyperplane in $\mathbb{R}^{d}$. Because
these hyperplane subsets are within the tile $P_{k}$, they have diameter
less than $4M$. The bound in 4 follows.
\end{proof}

We refer to a paritition $\mathcal{P}=\left\{ \varphi ^{w}C_{k}:w\in
Z_{k},1\leq k\leq K\right\} $ with the properties described above as a \emph{%
Voronoi-Rohlin partition} \emph{centered at }$C$. In the next section we
will use the notation $P_{k}$ in place of $Z_{k}$ for the integer partitions.

\subsection{The Ordered Group $\left( G\left( \protect\varphi \right)
,G\left( \protect\varphi \right) _{+}\right) $}

Given $\left( X,\varphi ,\mathbb{Z}^{d}\right) $ a minimal free Cantor
system, consider the group $C\left( X,\mathbb{Z}\right) $ under addition,
generated by indicator functions $1_{A}$ of clopen sets $A\subset X$. Let $%
B\left( \varphi \right) \subset C\left( X,\mathbb{Z}\right) $ denote the set
of $\varphi $\emph{-coboundaries}, i.e., functions which are sums of
functions of the form $1_{A}-1_{\varphi ^{v}A}$ where $A$ is clopen and $%
v\in \mathbb{Z}^{d}$.

\begin{definition}
\label{Gphi} We set: 
\begin{eqnarray*}
G\left( \varphi \right) &=&C\left( X,\mathbb{Z}\right) /B\left( \varphi
\right) \\
G\left( \varphi \right) _{+} &=&\left\{ \left[ f\right] :f\left( x\right)
\geq 0\text{ for all }x\in X\right\}
\end{eqnarray*}
\end{definition}

The pair $\left( G\left( \varphi \right) ,G\left( \varphi \right)
_{+}\right) $ is an \emph{ordered group}$\emph{,}$ that is, $G\left( \varphi
\right) $ is a countable abelian group, and $G\left( \varphi \right) _{+}$
is a subset of $G\left( \varphi \right) $ satisfying the following \cite%
{Forrest}.

\begin{enumerate}
\item $G\left( \varphi \right) _{+}+G\left( \varphi \right) _{+}\subset
G\left( \varphi \right) _{+}$

\item $G\left( \varphi \right) _{+}+\left( -G\left( \varphi \right)
_{+}\right) =G\left( \varphi \right) $

\item $G\left( \varphi \right) _{+}\cap \left( -G\left( \varphi \right)
_{+}\right) =0$
\end{enumerate}

\begin{notation}
For $g_{1}$ and $g_{2}$ in an ordered group $\left( G,G_{+}\right) $, we
will use the notation

\begin{itemize}
\item $g_{2}\geq g_{1}$ if $g_{2}-g_{1}\in G_{+}$

\item $g_{2}>g_{1}$ if $g_{2}\geq g_{1}$ and $g_{2}\neq g_{1}$.
\end{itemize}
\end{notation}

We need to employ Voronoi-Rohlin partitions in order to make use of the
ordered group $\left( G\left( \varphi \right) ,G\left( \varphi \right)
_{+}\right) $ as an invariant for bounded orbit injection equivalence. The
next two are the main lemmas along these lines.

\begin{lemma}
\label{lem:ineq}Let $\left( X,\varphi ,\mathbb{Z}^{d}\right) $ be a minimal
Cantor system, and let $f,g\in C\left( X,\mathbb{Z}\right) $. Suppose $\left[
f\right] <\left[ g\right] $ in $\left( G\left( \varphi \right) ,G\left(
\varphi \right) _{+}\right) $. Then there are constants $c>0$ and $M_{0}>2$
such that if $M>M_{0}$, $C$ is $M$-regular and $\mathcal{P}=\left\{ \varphi
^{w}C_{k}:w\in P_{k},1\leq k\leq K\right\} $ is a Voronoi-Rohlin partition
centered at $C$ then for $1\leq k\leq K$ and $x\in C_{k}$, 
\begin{equation*}
\sum_{v\in P_{k}}g\varphi ^{v}\left( x\right) -\sum_{v\in P_{k}}f\varphi
^{v}\left( x\right) >cM^{d}
\end{equation*}
\end{lemma}

\begin{proof}
We have that $g-f=h+\sum_{i=1}^{I}\left( 1_{A_{i}}-1_{\varphi
^{v(i)}A_{i}}\right) $ where $h\left( x\right) \geq 0$ for all $x\in X$, and
the $A_{i}$ are clopen. Since $\left[ f\right] <\left[ g\right] $ there is a
clopen set $U$ upon which $h\left( x\right) >0$. If $C$ is $M$-separated
then each $P_{k}$ contains all integer lattice points which are within $M/2$
of the origin. Because $U$ is $r$-syndetic for some $r>0$, we know that
there is a constant $\ell $ such that if $M$ is sufficiently large and $C$
is any $M$-separated set, then for all $x\in X$, 
\begin{equation*}
\sum_{v\in P_{k}}h\varphi ^{v}\left( x\right) >\ell M^{d}
\end{equation*}%
On the other hand for each $i$, 
\begin{eqnarray*}
\left\vert \sum_{v\in P_{k}}\left( 1_{A_{i}}-1_{\varphi ^{v(i)}A_{i}}\right)
\left( \varphi ^{v}\left( x\right) \right) \right\vert &=&\left\vert
\sum_{v\in P_{k}}1_{A_{i}}\left( \varphi ^{v}\left( x\right) \right)
-1_{A_{i}}\left( \varphi ^{v\mathbf{-}v\left( i\right) }\left( x\right)
\right) \right\vert \\
&\leq &\left\vert P_{k}\vartriangle \left( P_{k}-v\left( i\right) \right)
\right\vert
\end{eqnarray*}

By Theorem \ref{voronoi}, we have%
\begin{equation*}
\left\vert P_{k}\vartriangle \left( P_{k}-v\left( i\right) \right)
\right\vert \leq LM^{d-1}\left\Vert v\left( i\right) \right\Vert
\end{equation*}

If $M_{0}$ is large enough to insure $c=\ell -\frac{1}{M}L\sum_{i=1}^{I}%
\left\Vert v\left( i\right) \right\Vert >0$, then 
\begin{eqnarray*}
\sum_{v\in P_{k}}g\varphi ^{v}\left( x\right) -\sum_{v\in P_{k}}f\varphi
^{v}\left( x\right) &=&\sum_{v\in P_{k}}h\varphi ^{v}\left( x\right)
+\sum_{v\in P_{k}}\left( 1_{A_{i}}-1_{\varphi ^{v(i)}A_{i}}\right) \left(
\varphi ^{v}\left( x\right) \right) \\
&>&\ell M^{d}-\sum_{i=1}^{I}LM^{d-1}\left\Vert v\left( i\right) \right\Vert
\\
&=&cM^{d}
\end{eqnarray*}
\end{proof}

\begin{lemma}
\label{lem:ineq2}Let $\left( X,\varphi ,\mathbb{Z}^{d}\right) $ be a minimal
Cantor system, and let $f,g\in C\left( X,\mathbb{Z}\right) $. Suppose $%
\mathcal{P}=\left\{ \varphi ^{w}C_{k}:w\in P_{k},1\leq k\leq K\right\} $ is
a Voronoi-Rohlin partition centered at $C$ such that for all $x\in C_{k}$, 
\begin{equation*}
\sum_{v\in P_{k}}g\varphi ^{v}\left( x\right) \geq \sum_{v\in P_{k}}f\varphi
^{v}\left( x\right)
\end{equation*}%
then $\left[ g\right] \geq \left[ f\right] $. Moreover, if the above
inequality above holds for all $x\in C_{k}$, and for some $k$ and $x\in
C_{k} $ the inequality is strict then $\left[ g\right] >\left[ f\right] $.
\end{lemma}

\begin{proof}
Suppose $\mathcal{P}$ is a Voronoi-Rohlin partition satisfying the above
properties. For any function $h\in C\left( X,\mathbb{Z}\right) $, let $%
\widehat{h}\in C\left( X,\mathbb{Z}\right) $ be the following function 
\begin{equation*}
\widehat{h}\left( x\right) =\sum_{k=1}^{K}\sum_{w\in P_{k}}1_{C_{k}}\left(
x\right) \cdot h\varphi ^{w}\left( x\right)
\end{equation*}%
Now 
\begin{eqnarray*}
h-\widehat{h} &=&\sum_{k=1}^{K}\sum_{w\in P_{k}}1_{\varphi ^{w}C_{k}}\cdot
h-\sum_{k=1}^{K}\sum_{w\in P_{k}}1_{C_{k}}\cdot h\varphi ^{w} \\
&=&\sum_{k=1}^{K}\sum_{w\in P_{k}}\left( 1_{C_{k}}\varphi ^{-w}\cdot
h-1_{C_{k}}\cdot h\varphi ^{w}\right)
\end{eqnarray*}%
So $h-\widehat{h}\in B\left( \varphi \right) $ for any function $h\in
C\left( X,\mathbb{Z}\right) $.

The assumptions assert that $\widehat{g}-\widehat{f}$ is a nonnegative
function, which means that $\left[ g\right] \geq \left[ f\right] $. If in
addition, $\widehat{g}-\widehat{f}$ evaluates to a positive value at one $%
x\in X$, then $\left[ g\right] -\left[ f\right] \neq 0$, which implies $%
\left[ g\right] >\left[ f\right] $.
\end{proof}

In what follows we will use the following properties of $\left( G\left(
\varphi \right) ,G\left( \varphi \right) _{+}\right) $

\begin{enumerate}
\item $\left( G\left( \varphi \right) ,G\left( \varphi \right) _{+}\right) $
is \emph{weakly unperforated}, i.e., if $n\left[ f\right] >0$ with $n\in 
\mathbb{Z}$ and $\left[ f\right] \in G\left( \varphi \right) $ then $\left[ f%
\right] >0$.

\item $\left( G\left( \varphi \right) ,G\left( \varphi \right) _{+}\right) $
has the strong \emph{Reisz property}, i.e., if $\left[ g_{1}\right] $, $%
\left[ g_{2}\right] $, $\left[ f_{1}\right] $, $\left[ f_{2}\right] $ are in 
$G\left( \varphi \right) $ with $\left[ f_{j}\right] >\left[ g_{i}\right] $
for all $i,j\in \left\{ 1,2\right\} $ then there exists an $\left[ h\right]
\in G$ such that $\left[ f_{j}\right] >\left[ h\right] >\left[ g_{i}\right] $
for all $i,j\in \left\{ 1,2\right\} $.

\item $\left( G\left( \varphi \right) ,G\left( \varphi \right) _{+}\right) $
is \emph{simple}, i.e.,\emph{\ }if for every $\left[ f\right] \in G\left(
\varphi \right) _{+}\setminus \left\{ 0\right\} $ and $\left[ g\right] \in
G\left( \varphi \right) $, there is an $n\in \mathbb{N}$ such that $n\left[ f%
\right] >\left[ g\right] $.
\end{enumerate}

These properties were proven in \cite{Forrest}, along with the claim that $%
G\left( \varphi \right) $ is torsion-free. The torsion-freeness part of the
proof was later shown to be incorrect, leading to some interesting
developments, see \cite{Matui08}. Nevertheless, it is true that the above
properties hold for $\left( G\left( \varphi \right) ,G\left( \varphi \right)
_{+}\right) $; for example, they follow from Lemma \ref{lem:ineq2}.

\subsection{The ordered group as an invariant}

In the following, we use Voronoi-Rohlin partitions to prove that if: $\left(
X,\varphi ,\mathbb{Z}^{d}\right) $ and $\left( Y,\psi ,\mathbb{Z}^{d}\right) 
$ are bounded orbit injection equivalent then: 
\begin{equation*}
\left( G\left( \varphi \right) ,G\left( \varphi \right) _{+}\right) \cong
\left( G\left( \psi \right) ,G\left( \psi \right) _{+}\right) \text{.}
\end{equation*}

\begin{theorem}
\label{thm:iso}Let $d\geq 1$, and suppose $\left( X,\varphi ,\mathbb{Z}%
^{d}\right) $ and $\left( Y,\psi ,\mathbb{Z}^{d}\right) $ are bounded orbit
injection equivalent. Then $\left( G\left( \varphi \right) ,G\left( \varphi
\right) _{+}\right) \cong \left( G\left( \psi \right) ,G\left( \psi \right)
_{+}\right) $.
\end{theorem}

\begin{proof}
It suffices to consider the case where there is a bounded orbit injection $%
\theta :X\rightarrow Y$. Consider the homomorphism $h_{\theta }:C\left( X,%
\mathbb{Z}\right) \rightarrow G\left( \psi \right) $ defined on generators
by $h_{\theta }\left( 1_{A}\right) =\left[ 1_{\theta A}\right] $ for any
clopen set $A\subset X$.

To show that $h_{\theta }$ gives a well-defined homomorphism on $G\left(
\varphi \right) $, we wish to show that $h_{\theta }\left( 1_{A}-1_{\varphi
^{v}A}\right) =\left[ 0\right] $ for any clopen set $A\subset X$ and $v\in $ 
$\mathbb{Z}^{d}$. A clopen set $A$ partitions into finitely many clopen sets 
$A_{i}$ such that for each $i$, there is a $v\left( i\right) $ with $\theta
\varphi ^{v}\left( x\right) =\psi ^{v\left( i\right) }\theta \left( x\right) 
$ for all $x\in A_{i}$. Therefore, 
\begin{eqnarray*}
h\left( 1_{A}-1_{\varphi ^{v}A}\right) &=&\left[ 1_{\theta A}-1_{\theta
\varphi ^{v}A}\right] \\
&=&\sum_{i}\left[ 1_{\theta A_{i}}-1_{\theta \varphi ^{v}A_{i}}\right] \\
&=&\sum_{i}\left[ 1_{\theta A_{i}}-1_{\psi ^{v\left( i\right) }\theta A_{i}}%
\right] \\
&=&0
\end{eqnarray*}

Now we will construct an inverse homomorphism. Apply Proposition \ref{reg}
to $\theta \left( X\right) $ to obtain an $M$-regular subset $C$ of $\theta
\left( X\right) $ with $M>2$, and an associated Voronoi-Rohlin partition
centered at $C$, $\mathcal{P}=\left\{ \psi ^{w}C_{k}:w\in P_{k}\right\} $.
We can assume, after partitioning the clopen sets $C_{k}$ further if
necessary, that each partition element $\psi ^{w}C_{k}$ is a subset of $%
\theta \left( X\right) $ or $\theta \left( X\right) ^{c}$. From this, the
set $\theta \left( X\right) ^{c}$ is a disjoint union of clopen sets of the
form $\psi ^{w}C_{k}$, i.e., 
\begin{equation*}
\theta \left( X\right) ^{c}=\cup _{k=1}^{K}\cup _{j=1}^{J\left( k\right)
}\psi ^{w\left( j,k\right) }C_{k}
\end{equation*}%
where $w\left( j,k\right) \in P_{k}$ for all $\left( j,k\right) $. Let $f\in
C\left( Y,\mathbb{Z}\right) $. Then let $\widehat{f}\in C\left( Y,\mathbb{Z}%
\right) $ be the following function 
\begin{equation*}
\widehat{f}\left( x\right) =\left\{ 
\begin{tabular}{ll}
$0$ & if $x\in Y\setminus \theta \left( X\right) $ \\ 
$f\left( x\right) $ & if $x\in \theta \left( X\right) \setminus C$ \\ 
$f\left( x\right) +\sum_{k=1}^{K}\sum_{j=1}^{J\left( k\right)
}1_{C_{k}}\left( x\right) \cdot f\psi ^{w\left( j,k\right) }\left( x\right) $
& if $x\in C$%
\end{tabular}%
\right.
\end{equation*}%
It is not difficult to see that $f\mapsto \widehat{f}$ is a homomorphism,
and therefore the map $g_{\theta }:C\left( Y,\mathbb{Z}\right) \rightarrow
G\left( \varphi \right) $ defined by $g_{\theta }:f\mapsto \left[ \widehat{f}%
\theta \right] $ is a homomorphism. We wish to see that $g_{\theta }$
applied to a $\psi $-coboundary is equal to $\left[ 0\right] $ in $G\left(
\varphi \right) $. Recall that any $\psi $-coboundary is equal to the sum of
functions of the form $1_{B_{i}}-1_{\psi ^{v\left( i\right) }B_{i}}$ where
the sets $B_{i}\subset Y$ are clopen and vectors $v\left( i\right) \in $ $%
\mathbb{Z}^{d}$; note that the $B_{i}$ need not be distinct nor disjoint.
Further, by subdividing the sets if necessary, we may assume that each set $%
B_{i}$ and $\psi ^{v\left( i\right) }B_{i}$ is a subset of an element of the
Voronoi-Rohlin partition $\mathcal{P}=\left\{ \psi ^{w}C_{k}:w\in
P_{k},1\leq k\leq K\right\} $. Fix $i$ and set $B=B_{i}$, $v=v\left(
i\right) $. Then $B\subset \psi ^{w}C_{k}$ for some $w$, $k$. If $B\subset
\theta \left( X\right) $, then $\widehat{1}_{B}=1_{B}$ since $1_{B}\equiv 0$
on $Y\setminus \theta \left( X\right) $. If $B\subset Y\setminus \theta
\left( X\right) $, then $\widehat{1}_{B}=1_{C_{k}}\cdot 1_{B}\circ \psi
^{w}=1_{\psi ^{-w}B}$. In either case there is a vector $u$ such that $%
\widehat{1}_{B}=1_{\psi ^{u}B}$ and $\psi ^{u}B\subset \theta \left(
X\right) $. A similar fact is true of $\psi ^{v}B$, there is a vector $t$
such that $\widehat{1}_{\psi ^{v}B}=1_{\psi ^{t}B}$ with $\psi ^{t}B\subset
\theta \left( X\right) $. Setting $\widehat{B}=\psi ^{u}B$, and $\widehat{v}%
=t-u$ we have $\widehat{1}_{B}-\widehat{1}_{\psi ^{v}B}=1_{\widehat{B}%
}-1_{\psi ^{\widehat{v}}\widehat{B}}$.

Now let us consider $1_{\widehat{B}}\theta -1_{\psi ^{\widehat{v}}\widehat{B}%
}\theta =1_{\theta ^{-1}\widehat{B}}-1_{\theta ^{-1}\psi ^{\widehat{v}}%
\widehat{B}}$. Because $\widehat{B}$ and $\psi ^{\widehat{v}}\widehat{B}$
are subsets of $\theta \left( X\right) $, for each $x\in \theta ^{-1}%
\widehat{B}$ there is a vector $\alpha \left( x\right) $ such that $\theta
\varphi ^{\alpha \left( x\right) }\left( x\right) =\psi ^{\widehat{v}}\theta
\left( x\right) $. Further the function $\alpha $ is continuous and takes on
finitely many values $u\left( 1\right) ,u\left( 2\right) ,\ldots ,u\left(
I\right) $. Therefore, $1_{\theta ^{-1}\widehat{B}}-1_{\theta ^{-1}\psi ^{%
\widehat{v}}\widehat{B}}$ is a finite sum of the form $1_{A\left( i\right)
}-1_{\varphi ^{u\left( i\right) }A\left( i\right) }$ where $A\left( i\right) 
$ is the clopen set $A\left( i\right) =\left\{ x:\alpha \left( x\right)
=u\left( i\right) \right\} $. It follows that $g_{\theta }$ applied to a $%
\psi $-coboundary is a $\varphi $-coboundary. In particular this means that $%
g_{\theta }:G\left( \psi \right) \rightarrow G\left( \varphi \right) $ is
well-defined.

Now consider $g_{\theta }h_{\theta }$ applied to a function $1_{A}$ where $%
A\subset X$ is clopen. Then since $\theta \left( A\right) \subset \theta
\left( X\right) $, $\widehat{1}_{\theta A}=1_{\theta A}$. Further since $%
\theta $ is injective, $\theta ^{-1}\theta A=A$. Thus, 
\begin{eqnarray*}
g_{\theta }h_{\theta }\left[ 1_{A}\right] &=&g_{\theta }\left[ 1_{\theta A}%
\right] \\
&=&\left[ \widehat{1}_{\theta A}\theta \right] \\
&=&\left[ 1_{\theta A}\theta \right] \\
&=&\left[ 1_{\theta ^{-1}\theta A}\right] \\
&=&\left[ 1_{A}\right]
\end{eqnarray*}

Now consider $h_{\theta }g_{\theta }$ applied to a function $1_{B}$ where $%
B\subset Y$ is clopen and $B\subset \psi ^{w}C_{k}$ for some $w$, $k$. Then
recall that $\widehat{1}_{B}=1_{\psi ^{u}B}$ for some $u$ with $\psi
^{u}B\subset \theta \left( X\right) $%
\begin{eqnarray*}
h_{\theta }g_{\theta }\left[ 1_{B}\right] &=&h_{\theta }\left[ \widehat{1}%
_{B}\theta \right] \\
&=&h_{\theta }\left[ 1_{\psi ^{u}B}\theta \right] \\
&=&h_{\theta }\left[ 1_{\theta ^{-1}\psi ^{u}B}\right] \\
&=&\left[ 1_{\theta \theta ^{-1}\psi ^{u}B}\right] \\
&=&\left[ 1_{\psi ^{u}B}\right] \\
&=&\left[ 1_{B}\right]
\end{eqnarray*}%
Therefore, $g_{\theta }h_{\theta }$ and $h_{\theta }g_{\theta }$ are both
identity maps and $G\left( \varphi \right) $, $G\left( \psi \right) $ are
isomorphic.

To see that the positive cones are preserved, consider $f\in C\left( X,%
\mathbb{Z}\right) $ with $f\left( x\right) \geq 0$ for all $x\in X$. Then $%
f=\sum c_{i}1_{A_{i}}$ where $c_{i}>0$, $A_{i}$ are clopen. $h_{\theta }%
\left[ f\right] =\left[ \sum c_{i}1_{\theta A_{i}}\right] \in G\left( \psi
\right) _{+}$.

Conversely, suppose $f\in C\left( Y,\mathbb{Z}\right) $ and $f\left(
y\right) \geq 0$ for all $y\in Y$. Then $\widehat{f}\left( y\right) \geq 0$
for all $y\in Y$, and $\widehat{f}\theta \left( x\right) \geq 0$ for all $%
x\in X$. Thus $g_{\theta }\left[ f\right] \in G\left( \varphi \right) _{+}$.
\end{proof}

Suppose $\left( X,\varphi ,\mathbb{Z}^{d}\right) $ and $\left( Y,\psi ,%
\mathbb{Z}^{d}\right) $ are minimal Cantor systems which are bounded orbit
injection equivalent by virtue of bounded orbit injections into a common
system $\left( Z,\alpha ,\mathbb{Z}^{d}\right) $. Below we give a condition
on the isomorphisms created in Theorem \ref{thm:iso} which guarantees that
there is a bounded orbit injection from $\left( X,\varphi ,\mathbb{Z}%
^{d}\right) $ into $\left( Y,\psi ,\mathbb{Z}^{d}\right) $. This in turn,
leads to a proof that bounded orbit injection equivalence is in fact an
equivalence relation for $d>2$ (the cases where $d=1,2$ are already covered
by \cite{Lightwood07}).

\begin{theorem}
\label{small}Suppose there exist bounded orbit injections $\theta _{1}$ and $%
\theta _{2}$ from systems $\left( X,\varphi ,\mathbb{Z}^{d}\right) $ and $%
\left( Y,\psi ,\mathbb{Z}^{d}\right) $ into $\left( Z,\alpha ,\mathbb{Z}%
^{d}\right) $, and that $h_{\theta _{1}}\left[ 1_{X}\right] <h_{\theta _{2}}%
\left[ 1_{Y}\right] $. Then there is a bounded orbit injection from $\left(
X,\varphi ,\mathbb{Z}^{d}\right) $ into $\left( Y,\psi ,\mathbb{Z}%
^{d}\right) $.
\end{theorem}

\begin{proof}
By Lemma \ref{lem:ineq}, there is a Voronoi-Rohlin partition:%
\begin{equation*}
\mathcal{P}=\left\{ \alpha ^{v}C_{k}:w\in P_{k},1\leq k\leq K\right\}
\end{equation*}
of $Z$ centered at $C=\cup _{k=1}^{K}C_{k}$ such that for $x\in C_{k}$, 
\begin{equation*}
\sum_{v\in P_{k}}1_{\theta _{2}\left( Y\right) }\alpha ^{v}\left( x\right)
>\sum_{v\in P_{k}}1_{\theta _{1}\left( X\right) }\alpha ^{v}\left( x\right)
\end{equation*}%
By further refining if necessary, we may assume $\mathcal{P}$ refines both $%
\left\{ \theta _{1}\left( X\right) ,Z\setminus \theta _{1}\left( X\right)
\right\} $ and $\left\{ \theta _{2}\left( Y\right) ,Z\setminus \theta
_{2}\left( Y\right) \right\} $. Now for each $k$, we define an injection 
\begin{equation*}
\rho _{k}:\left\{ v\in P_{k}:\alpha ^{v}C_{k}\subset \theta _{1}\left(
X\right) \right\} \rightarrow \left\{ v\in P_{k}:\alpha ^{v}C_{k}\subset
\theta _{2}\left( Y\right) \right\}
\end{equation*}%
For $x\in \alpha ^{v}C_{k}\cap \theta _{1}\left( X\right) \in \mathcal{P}$
set $\pi \left( x\right) =\alpha ^{\rho _{k}\left( v\right) -v}\left(
x\right) $. Then $\pi \left( \theta _{1}\left( X\right) \right) \subset
\theta _{2}\left( Y\right) $ and we have a bounded orbit injection from $%
\left( X,\varphi ,\mathbb{Z}^{d}\right) $ into $\left( Y,\psi ,\mathbb{Z}%
^{d}\right) $ defined by $\theta _{2}^{-1}\pi \theta _{1}$.
\end{proof}

Let $m>0$. By the \emph{tower of height }$\emph{m}$\emph{\ over }$\left(
X,\varphi ,\mathbb{Z}^{d}\right) $, we mean the system of the form $\left(
X\left( m\right) ,\widehat{\varphi },\mathbb{Z}^{d}\right) $ where $X\left(
m\right) =X\times \left\{ 0,1,\ldots ,m-1\right\} ^{d}$ and $\widehat{%
\varphi }$ is defined by the following. Let $x\in X$, $v\in \mathbb{Z}^{d}$
and $u\in \left\{ 0,1,\ldots ,m-1\right\} ^{d}$. Then $v+u$ can be written
in the form $mw+r$ where $w\in \mathbb{Z}^{d}$ and $r\in \left\{ 0,1,\ldots
,m-1\right\} ^{d}$. We then define $\widehat{\varphi }^{v}\left( x,u\right)
=\left( \varphi ^{w}\left( x\right) ,r\right) $. Note that $X\times \left\{
0,1,\ldots ,m-1\right\} ^{d}$ is a Cantor set, and that if $\varphi $ is a
minimal free $\mathbb{Z}^{d}$-action of $X$, then $\widehat{\varphi }$ is a
minimal free $\mathbb{Z}^{d}$-action of $X\left( m\right) $. Further note
that there is a bounded orbit injection from $\left( X,\varphi ,\mathbb{Z}%
^{d}\right) $ into a tower of height $m$ over $\left( X,\varphi ,\mathbb{Z}%
^{d}\right) $ given by $\theta \left( x\right) =\left( x,0\right) $. In
this, the isomorphism $h_{\theta }:G\left( \varphi \right) \rightarrow
G\left( \widehat{\varphi }\right) $ satisfies $m^{d}h_{\theta }\left[ 1_{X}%
\right] =\left[ 1_{X\left( m\right) }\right] $.

\begin{lemma}
\label{tower}Suppose there exist bounded orbit injections $\theta _{1}$ and $%
\theta _{2}$ from systems $\left( X,\varphi ,\mathbb{Z}^{d}\right) $ and $%
\left( Y,\psi ,\mathbb{Z}^{d}\right) $ into $\left( Z,\alpha ,\mathbb{Z}%
^{d}\right) $. Then for some $m>0$ there is a bounded orbit injection from $%
\left( X,\varphi ,\mathbb{Z}^{d}\right) $ into a tower of height $m$ over $%
\left( Y,\psi ,\mathbb{Z}^{d}\right) $.
\end{lemma}

\begin{proof}
By the fact that $G\left( \alpha \right) $ is simple, there is an $n>0$ such
that $h_{\theta _{1}}\left[ 1_{X}\right] <nh_{\theta _{2}}\left[ 1_{Y}\right]
$. Fix $m>0$ so that $m^{d}>n$. Notice that there is a bounded orbit
injection $\widehat{\theta }_{1}$ from $\left( X,\varphi ,\mathbb{Z}%
^{d}\right) $ into a tower of height $m$ over $\left( Z,\alpha ,\mathbb{Z}%
^{d}\right) $ defined by $\widehat{\theta }_{1}\left( x\right) =\left(
\theta _{1}\left( x\right) ,0\right) $. Also notice that there is a bounded
orbit injection $\widehat{\theta }_{2}$ from the tower of height $m$ over $%
\left( Y,\psi ,\mathbb{Z}^{d}\right) $ into the tower of height $m$ over $%
\left( Z,\alpha ,\mathbb{Z}^{d}\right) $, defined by $\widehat{\theta }%
_{2}\left( y,u\right) =\left( \theta _{2}\left( y\right) ,u\right) $ for $%
u\in \left\{ 0,1,\ldots ,m-1\right\} ^{d}$. Now 
\begin{equation*}
h_{\widehat{\theta }_{1}}\left[ 1_{X}\right] <nh_{\widehat{\theta }_{2}}%
\left[ 1_{Y}\right] <m^{d}h_{\widehat{\theta }_{2}}\left[ 1_{Y}\right] =h_{%
\widehat{\theta }_{2}}\left[ 1_{Y\left( m\right) }\right]
\end{equation*}
Therefore, by Lemma \ref{small}, there is a bounded orbit injection from $%
\left( X,\varphi ,\mathbb{Z}^{d}\right) $ into $\left( Y\left( m\right) ,%
\widehat{\psi },\mathbb{Z}^{d}\right) $.
\end{proof}

\begin{theorem}
Bounded orbit injection equivalence is an equivalence relation.
\end{theorem}

\begin{proof}
Reflexivity and symmetry are clear, we prove transitivity. Suppose there are
bounded orbit injections from $\left( X,\varphi ,\mathbb{Z}^{d}\right) $ and 
$\left( Y,\psi ,\mathbb{Z}^{d}\right) $ into a common system and there are
bounded orbit injections from $\left( Y,\psi ,\mathbb{Z}^{d}\right) $ and $%
\left( Z,\alpha ,\mathbb{Z}^{d}\right) $ into a common system. Then by Lemma %
\ref{tower}, there is an $m>0$ such that there exist bounded orbit
injections from both $\left( X,\varphi ,\mathbb{Z}^{d}\right) $ and $\left(
Z,\alpha ,\mathbb{Z}^{d}\right) $ into a tower of height $m$ over $\left(
Y,\psi ,\mathbb{Z}^{d}\right) $.
\end{proof}

\subsection{Orbit equivalences from injections}

Now let us suppose that two minimal Cantor systems $\left( X,\varphi ,%
\mathbb{Z}^{d}\right) $ and $\left( Y,\psi ,\mathbb{Z}^{d}\right) $ are
bounded orbit injection equivalent, with bounded orbit injections $\theta
_{1}$ and $\theta _{2}$ into a common system $\left( Z,\alpha ,\mathbb{Z}%
^{d}\right) $. Let $h_{\theta _{1}}$ and $h_{\theta _{2}}$ be the
isomorphisms as in Theorem \ref{thm:iso}. We show that if $h_{\theta _{1}}%
\left[ 1_{X}\right] =h_{\theta _{2}}\left[ 1_{Y}\right] $ holds then there
is a \emph{bounded orbit equivalence} $\theta $ from $\left( X,\varphi ,%
\mathbb{Z}^{d}\right) $ to $\left( Y,\psi ,\mathbb{Z}^{d}\right) $, i.e., an
orbit equivalence $\theta $ in which the function $\eta :X\times \mathbb{Z}%
^{d}\rightarrow \mathbb{Z}^{d}$ satisfying $\theta \left( \varphi ^{v}\left(
x\right) \right) =\psi ^{\eta \left( x,v\right) }\theta \left( x\right) $ is
continuous.

Let $\left( X,\varphi ,\mathbb{Z}^{d}\right) $ be a minimal Cantor system.
By the \emph{full group of} $\varphi $ we mean the collection of
homeomorphisms $\pi :X\rightarrow X$ such that for each $x$, $\pi \left(
x\right) =\varphi ^{\zeta \left( x\right) }\left( x\right) $ for some $\zeta
\left( x\right) \in \mathbb{Z}^{d}$.\ By the \emph{topological full group of 
}$\varphi $ we mean the collection of full group elements that have the
property that the associated cocycle function $\zeta :X\rightarrow \mathbb{Z}%
^{d}$ is continuous.

\begin{theorem}
Let $\left( X,\varphi ,\mathbb{Z}^{d}\right) $ and $\left( Y,\psi ,\mathbb{Z}%
^{d}\right) $ be two minimal Cantor systems which are bounded orbit
injection equivalent, with bounded orbit injections $\theta _{1}$ and $%
\theta _{2}$ into a common system $\left( Z,\alpha ,\mathbb{Z}^{d}\right) $.
Let $h_{\theta _{1}}:\left( G\left( \varphi \right) ,G\left( \varphi \right)
_{+}\right) \rightarrow \left( G\left( \alpha \right) ,G\left( \alpha
\right) _{+}\right) $ and $h_{\theta _{2}}:\left( G\left( \psi \right)
,G\left( \psi \right) _{+}\right) \rightarrow \left( G\left( \alpha \right)
,G\left( \alpha \right) _{+}\right) $ be the isomorphisms as in Theorem \ref%
{thm:iso}. Suppose $h_{\theta _{1}}\left[ 1_{X}\right] =h_{\theta _{2}}\left[
1_{Y}\right] $ in $G\left( \alpha \right) $. Then there is a bounded orbit
equivalence $\theta $ from $\left( X,\varphi ,\mathbb{Z}^{d}\right) $ to $%
\left( Y,\psi ,\mathbb{Z}^{d}\right) $. Further, $h_{\theta }=h_{\theta
_{2}}^{-1}h_{\theta _{1}}$ is an isomorphism from $\left( G\left( \varphi
\right) ,G\left( \varphi \right) _{+}\right) \ $to $\left( G\left( \psi
\right) ,G\left( \psi \right) _{+}\right) $.
\end{theorem}

\begin{proof}
First consider the case where $\theta _{1}\left( X\right) =Z$. Then $\theta
_{2}\left( Y\right) =Z$ as well, for otherwise $Z\setminus \theta _{2}\left(
Y\right) $ is a nonempty clopen set and $h_{\theta _{1}}\left[ 1_{X}\right] =%
\left[ 1_{Z}\right] =h_{\theta _{2}}\left[ 1_{Y}\right] +\left[
1_{Z\setminus \theta _{2}\left( Y\right) }\right] >h_{\theta _{2}}\left[
1_{Y}\right] $. But if $\theta _{1}\left( X\right) =Z=\theta _{2}\left(
Y\right) $ then we can simply let $\theta =\theta _{2}^{-1}\theta _{1}$, and
we are done.

Assume then that $\left[ 1_{Z}\right] >\left[ 1_{\theta _{1}\left( X\right) }%
\right] =\left[ 1_{\theta _{2}\left( X\right) }\right] $. Then 
\begin{equation*}
1_{\theta _{1}\left( X\right) }-1_{\theta _{2}\left( Y\right)
}=\sum_{i=1}^{I}\left( 1_{A_{i}}-1_{\alpha ^{v\left( i\right) }A_{i}}\right)
\end{equation*}%
where $A_{i}\subset Z$ are clopen and $v\left( i\right) \in \mathbb{Z}^{d}$.
Via Lemma \ref{lem:ineq}, there is a $c>0$ and $M_{0}>2$ such that if $C$ is 
$M$-regular and $\mathcal{P}=\left\{ \varphi ^{w}C_{k}:w\in P_{k},1\leq
k\leq K\right\} $ is a Voronoi-Rohlin partition centered at $C$ then for $%
1\leq k\leq K$ and $x\in C_{k}$, both the following hold 
\begin{eqnarray*}
\sum_{v\in P_{k}}1_{\theta _{2}\left( Y\right) }\alpha ^{v}\left( x\right)
&>&cM^{d} \\
\sum_{v\in P_{k}}1_{Z}\alpha ^{v}\left( x\right) -\sum_{v\in P_{k}}1_{\theta
_{2}\left( Y\right) }\alpha ^{v}\left( x\right) &>&cM^{d}
\end{eqnarray*}

Choose $M$ such that $cM>L\sum_{i=1}^{I}\left\Vert v\left( i\right)
\right\Vert $ where $L$ is the constant from Theorem \ref{voronoi} (item 4).

Let $\mathcal{P}=\left\{ \alpha ^{w}C_{k}:w\in P_{k},1\leq k\leq K\right\} $
be a Voronoi-Rohlin partition of $Z$ centered at an $M$-regular clopen set $%
C $. If necessary, partition the clopen sets $C_{k}$ so that $\mathcal{P}$
refines $\left\{ A_{i},Z\setminus A_{i}\right\} $ and $\left\{ \alpha
^{v\left( i\right) }A_{i},Z\setminus \alpha ^{v\left( i\right)
}A_{i}\right\} $ for each $i$.

Then for each $i$, 
\begin{equation*}
A_{i}=\cup _{k=1}^{K}\cup _{j=1}^{J\left( k\right) }\alpha ^{w\left(
i,j,k\right) }C_{k}
\end{equation*}%
where $w\left( i,j,k\right) \in P_{k}$ for all $\left( i,j,k\right) $. Let $%
B_{i}\subset A_{i}$ be the union of sets $\alpha ^{w\left( i,j,k\right)
}C_{k}$ over the indices $\left( i,j,k\right) $ where $w\left( i,j,k\right)
+v\left( i\right) \not\in P_{k}$. The number of such indices $\left(
i,j,k\right) $ is bounded above by 
\begin{eqnarray*}
\sum_{i=1}^{I}\left\vert P_{k}\vartriangle \left( P_{k}-v\left( i\right)
\right) \right\vert &<&\sum_{i=1}^{I}L\left\Vert v\left( i\right)
\right\Vert M^{d-1} \\
&<&cM^{d}
\end{eqnarray*}

For each $x\in B_{i}$, $x\in \alpha ^{w\left( x\right) }C_{k\left( x\right)
} $ for some $1\leq l\left( x\right) \leq K$ and $w\left( x\right) \in
P_{k\left( x\right) }$ and $\alpha ^{v\left( i\right) }\left( x\right) $ is
in $\alpha ^{u\left( x\right) }C_{l\left( x\right) }$ for some $1\leq
l\left( x\right) \leq K$ for some $u\left( x\right) \in P_{l\left( x\right)
} $. The map $x\mapsto \left( k\left( x\right) ,w\left( x\right) ,l\left(
x\right) ,u\left( x\right) \right) $ is continuous.

For each pair of indices $k$ and $w$ with $\alpha ^{w}C_{k}\subset B_{i}$,
we select a vector $r\left( i,w,k\right) \in P_{k}$ so that $\alpha
^{r\left( i,w,k\right) }C_{k}\subset Z\setminus \theta _{2}\left( Y\right) $%
. We do so in such a way that if $\left( i,w\right) \neq \left( i^{\prime
},w^{\prime }\right) $ then $r\left( i,w,k\right) \neq r\left( i^{\prime
},w^{\prime },k\right) $. This is possible because 
\begin{eqnarray*}
\#\left\{ \left( i,k,w\right) :\alpha ^{w}C_{k}\subset B_{i}\text{ for some }%
i\right\} &=&\sum_{i=1}^{I}\left\vert P_{k}\vartriangle \left( P_{k}-v\left(
i\right) \right) \right\vert \\
&<&cM^{d} \\
&<&\sum_{v\in P_{k}}1_{Z\setminus \theta _{2}\left( Y\right) }\alpha
^{v}\left( x\right)
\end{eqnarray*}

By the same reasoning, for each index $l$ and $u$ with $\alpha
^{u}C_{l}\subset \alpha ^{v\left( i\right) }B_{i}$, we select a vector $%
s\left( i,u,l\right) \in P_{l}$ so that $\alpha ^{s\left( i,u,l\right)
}C_{l}\subset \theta _{2}\left( Y\right) $, and that if $\left( i,u\right)
\neq \left( i^{\prime },u^{\prime }\right) $ then $s\left( i,u,l\right) \neq
s\left( i^{\prime },u^{\prime },l\right) $.

Define $\pi :\theta _{2}\left( Y\right) \rightarrow Z$, an element of the
topological full group of $\alpha $ as follows. Suppose $y=\alpha ^{v\left(
i\right) }\left( x\right) $ for some $x\in B_{i}$ where $B_{i}\subset \alpha
^{w}C_{k}\in \mathcal{P}$, and $y\in \alpha ^{u}C_{l}\in \mathcal{P}$.
Define $\pi \left( \alpha ^{s\left( i,u,l\right) -u}y\right) =\alpha
^{r\left( i,w,k\right) -w}\left( x\right) $. We set $\pi \equiv id$ on the
complement of $\cup _{i,u,l}\alpha ^{s\left( i,u,l\right) -u}C_{l}$.

Now replace the bounded orbit injection $\theta _{2}:Y\rightarrow Z$ with
the bounded orbit injection $\pi \theta _{2}:Y\rightarrow Z$.

Fix $k$, and let us consider $z\in C_{k}$ and the sum 
\begin{eqnarray*}
\sum_{v\in P_{k}}\left( 1_{\theta _{1}\left( X\right) }\left( \alpha
^{v}z\right) -1_{\pi \theta _{2}\left( Y\right) }\left( \alpha ^{v}z\right)
\right) &=&\sum_{v\in P_{k}}\left( 1_{\theta _{1}\left( X\right) }\left(
\alpha ^{v}z\right) -1_{\theta _{2}\left( Y\right) }\left( \alpha
^{v}z\right) \right) \\
&&+\sum_{v\in P_{k}}\left( 1_{\theta _{2}\left( Y\right) }\left( \alpha
^{v}z\right) -1_{\pi \theta _{2}\left( Y\right) }\left( \alpha ^{v}z\right)
\right)
\end{eqnarray*}%
We rewrite the first and second terms in the sum as follows.%
\begin{eqnarray*}
\sum_{v\in P_{k}}\left( 1_{\theta _{1}\left( X\right) }\left( \alpha
^{v}z\right) -1_{\theta _{2}\left( Y\right) }\left( \alpha ^{v}z\right)
\right) &=&\sum_{v\in P_{k}}\sum_{i}\left( 1_{B_{i}}\left( \alpha
^{v}z\right) -1_{\alpha ^{v\left( i\right) }B_{i}}\left( \alpha ^{v}z\right)
\right) \\
&=&\sum_{i}\#\left\{ v\in P_{k}:\alpha ^{v}C_{k}\subset B_{i}\right\} \\
&&-\sum_{i}\#\left\{ v\in P_{k}:\alpha ^{v-v\left( i\right) }C_{k}\subset
B_{i}\right\}
\end{eqnarray*}%
\begin{eqnarray*}
&&\sum_{v\in P_{k}}\left( 1_{\theta _{2}\left( Y\right) }\left( \alpha
^{v}z\right) -1_{\pi \theta _{2}\left( Y\right) }\left( \alpha ^{v}z\right)
\right) = \\
&&\sum_{i}\#\left\{ v\in P_{k}:v=s\left( i,u,k\right) \text{ for some }u\in
P_{k}\right\} \\
&&-\sum_{i}\#\left\{ v\in P_{k}:r\left( i,w,k\right) =v\text{ for some }w\in
P_{k}\right\}
\end{eqnarray*}

Fixing $i$, these sums cancel. This means that for each $z\in C_{k}$, there
is a one-to-one correspondence $\zeta :\left\{ v\in P_{k}:\alpha ^{v}\left(
z\right) \subset \theta _{1}\left( X\right) \right\} \rightarrow \left\{
v\in P_{k}:\alpha ^{v}\left( z\right) \subset \pi \theta _{2}\left( Y\right)
\right\} $.

With this, we can set up a bounded orbit equivalence $h$ from $X$ to $Y$ by
taking $h\left( x\right) =\theta _{2}^{-1}\pi ^{-1}\alpha ^{\zeta \left(
v\right) -v}\theta _{1}\left( x\right) $ where $\theta _{1}\left( x\right)
\in \alpha ^{v}C_{k}$.
\end{proof}

Finally we examine the situation where two minimal Cantor systems $\left(
X,\varphi ,\mathbb{Z}^{d}\right) $ and $\left( Y,\psi ,\mathbb{Z}^{d}\right) 
$ are bounded orbit injection equivalent, with bounded orbit injections $%
\theta _{1}$ and $\theta _{2}$ into a common system $\left( Z,\alpha ,%
\mathbb{Z}^{d}\right) $ with the property that $h_{\theta _{1}}\left[ 1_{X}%
\right] -h_{\theta _{2}}\left[ 1_{Y}\right] $ is an infinitesimal, i.e., an
element of the subgroup $Inf\left( G\left( \alpha \right) \right) $ as
defined below.

\begin{definition}
For a simple ordered group $\left( G,G_{+}\right) $ we define $Inf\left(
G\right) $ to be the following subgroup 
\begin{equation*}
Inf\left( G\right) =\left\{ g\in G:ng<h\text{ for any }n\in \mathbb{Z}\text{%
, and any }h\in G_{+}\setminus \left\{ 0\right\} \right\}
\end{equation*}
\end{definition}

The type of bounded orbit injection equivalence discussed here corresponds
to a kind of topological version of the notion of \emph{even }Kakutani
equivalence in measure theoretic dynamics (see for example, \cite{Rudolph}).
It follows from the main results of \cite{Putnam09} that $\left( X,\varphi ,%
\mathbb{Z}^{d}\right) $ and $\left( Y,\psi ,\mathbb{Z}^{d}\right) $ are
orbit equivalent. We obtain a stronger form of orbit equivalence in this
setting, not bounded, but where for any $v\in \mathbb{Z}^{d}$ the cocycle $%
\eta \left( \cdot ,v\right) $ is continuous except at two points $\left\{
x_{0},\varphi ^{-v}x_{0}\right\} $.

We first show over the next three propositions that if indicator functions
of clopen sets differ by a infinitesimal then there is a full group element
mapping one to the other. The proof is essentially an adaptation of an
argument in \cite{Glasner} to the case of $\mathbb{Z}^{d}$-actions.

\begin{proposition}
Let $\left( X,\varphi ,\mathbb{Z}^{d}\right) $ be a minimal Cantor system.
Let $A$, $B$ be two clopen sets in $X$ such that $\left[ 1_{A}\right] -\left[
1_{B}\right] \in Inf\left( \varphi \right) $. Then given a proper clopen
subset $D\subset A$ and point $b\in B$, there is a clopen subset $E\subset
B\setminus \left\{ b\right\} $ such that $\left[ 1_{D}\right] <\left[ 1_{E}%
\right] $.
\end{proposition}

\begin{proof}
Let $\varepsilon =\left[ 1_{A}\right] -\left[ 1_{B}\right] \in Inf\left(
\varphi \right) $. Since $D\subset A$ is a proper subset, $A\setminus D$ is
a nonempty clopen set and $\left[ 1_{A\setminus D}\right] >0$. Now $\left[
1_{B}\right] +\varepsilon -\left[ 1_{D}\right] =\left[ 1_{A}\right] -\left[
1_{D}\right] =\left[ 1_{A\setminus D}\right] >0$. Because $\left( G\left(
\varphi \right) ,G\left( \varphi \right) _{+}\right) $ is simple, there
exists an $n\in \mathbb{N}$ such that $n\left( \left[ 1_{B}\right]
+\varepsilon -\left[ 1_{D}\right] \right) -\left[ 1\right] >0$. Because $%
\varepsilon $ is infinitesimal, $\left[ 1\right] -n\varepsilon >0$, adding
this to $n\left( \left[ 1_{B}\right] +\varepsilon -\left[ 1_{D}\right]
\right) -\left[ 1\right] $, we obtain $n\left( \left[ 1_{B}\right] -\left[
1_{D}\right] \right) >0$. Because the ordered group is weakly unperforated, $%
\left[ 1_{B}\right] -\left[ 1_{D}\right] >0$.

Now use Lemma \ref{lem:ineq} to find a $M_{0}$ such that whenever $C$ is $M$%
-regular for $M>M_{0}$ and $\mathcal{P}=\left\{ \varphi ^{w}C_{k}:w\in
P_{k},1\leq k\leq K\right\} $ is a Voronoi-Rohlin partition of $X$ centered
at $C$ then for all $x\in C_{k}$, 
\begin{equation*}
\sum_{v\in P_{k}}1_{D}\varphi ^{v}\left( x\right) -\sum_{v\in
P_{k}}1_{B}\varphi ^{v}\left( x\right) >2
\end{equation*}

Apply Proposition \ref{reg} to construct an $M$-regular clopen set $C\subset
B$ containing $b$ with $M>M_{0}$ and the corresponding Voronoi-Rohlin
partition 
\begin{equation*}
\mathcal{P}=\left\{ \varphi ^{w}C_{k}:w\in P_{k},1\leq k\leq K\right\} \text{%
.}
\end{equation*}%
By partitioning $C_{k}$ if necessary, we may assume that $\mathcal{P}$ is a
finer partition than both $\left\{ D,X\setminus D\right\} $ and $\left\{
B,X\setminus B\right\} $. One of these partition elements $B$ is a union of
the form $\cup _{k=1}^{K}\cup _{j=1}^{J\left( k\right) }\varphi ^{w\left(
j,k\right) }C_{k}$ and one of these sets, say $\varphi ^{w\left( 1,1\right)
}C_{1}$, contains the point $b$. Let $E$ be the clopen set $B\setminus
\varphi ^{w\left( 1,1\right) }C_{1}$. Then $E\subset B$, $b\not\in E$, and%
\begin{equation*}
\sum_{v\in P_{k}}1_{D}\varphi ^{v}\left( x\right) <\sum_{v\in
P_{k}}1_{E}\varphi ^{v}\left( x\right) \text{.}
\end{equation*}%
Therefore, $\left[ 1_{E}\right] -\left[ 1_{D}\right] \in G\left( \varphi
\right) _{+}\setminus \left[ 0\right] $ which gives the result.
\end{proof}

\begin{lemma}
\label{step}Let $\left( X,\varphi ,\mathbb{Z}^{d}\right) $ be a minimal
Cantor system. Let $A$, $B$ be two clopen sets in $X$ with $\left[ 1_{A}%
\right] -\left[ 1_{B}\right] \in Inf\left( \varphi \right) $. Fix $x_{0}\in
A $, $y_{0}\in B$ and let $\epsilon >0$ be given. Then there is an element $%
\pi $ of the topological full group of $\varphi $ and a clopen set $%
A^{\prime }\subset A\setminus \left\{ x_{0}\right\} $ such that

\begin{enumerate}
\item $A^{\prime }\supset A\setminus B\left( x_{0},\epsilon \right) $

\item $\pi \left( A^{\prime }\right) \subset B\setminus \left\{
y_{0}\right\} $

\item $\pi ^{2}=id$

\item $\pi |_{X\setminus A^{\prime }}=id$
\end{enumerate}
\end{lemma}

\begin{proof}
Without loss of generality, $A$ and $B$ are disjoint, otherwise set $\pi =id$
on $A\cap B$. Let $A^{\prime }$ be any clopen set such that $A\setminus
B\left( x_{0},\epsilon \right) \subset A^{\prime }\subset A\setminus \left\{
x_{0}\right\} $. Then by the previous proposition there is a clopen set $%
B^{\prime }\subset B\setminus \left\{ y_{0}\right\} $ such that $\left[
1_{B^{\prime }}\right] >\left[ 1_{A^{\prime }}\right] $. Now use Lemma \ref%
{lem:ineq} to find a $M_{0}$ such that whenever $C$ is $M$-regular for $%
M>M_{0}$ and $\mathcal{P}=\left\{ \varphi ^{w}C_{k}:w\in P_{k},1\leq k\leq
K\right\} $ is a Voronoi-Rohlin partition of $X$ centered at $C$ then for
all $x\in C$, 
\begin{equation*}
\sum_{v\in P_{k}}1_{A^{\prime }}\varphi ^{v}\left( x\right) <\sum_{v\in
P_{k}}1_{B^{\prime }}\varphi ^{v}\left( x\right)
\end{equation*}

Next apply Proposition \ref{reg} to construct an $M$-regular clopen set $%
C\subset A$ containing $x_{0}$ with $M>M_{0}$, and the corresponding
Voronoi-Rohlin partition $\mathcal{P}=\left\{ \varphi ^{w}C_{k}:w\in
P_{k},1\leq k\leq K\right\} $. By partitioning $C_{k}$ if necessary, we may
assume that $\mathcal{P}$ is a finer partition than both $\left\{ A^{\prime
},X\setminus A^{\prime }\right\} $ and $\left\{ B^{\prime },X\setminus
B^{\prime }\right\} $. Then for each $k$ we can define an injection $\pi
_{k}:\left\{ v\in P_{k}:\varphi ^{v}C_{k}\subset A^{\prime }\right\}
\rightarrow \left\{ v\in P_{k}:\varphi ^{v}C_{k}\subset B^{\prime }\right\} $%
.\ Then for $x\in A^{\prime }$ we know $x\in $ $\varphi ^{v}C_{k}$ for some $%
k$ and some $v\in P_{k}$; define $\pi \left( x\right) =\varphi ^{\pi
_{k}\left( v\right) -v}\left( x\right) $. Then we see that $\pi \left(
x\right) \in $ $\varphi ^{\pi _{k}\left( v\right) -v}\varphi
^{v}C_{k}=\varphi ^{\pi _{k}\left( v\right) }C_{k}\subset B^{\prime }$.

To extend $\pi $ to a homeomorphism of $X$, we set $\pi =\pi ^{-1}$ on $%
B^{\prime }$ and $\pi =id$ elsewhere.
\end{proof}

\begin{theorem}
\label{fginf}Let $\left( X,\varphi ,\mathbb{Z}^{d}\right) $ be a minimal
Cantor system. Let $A$, $B$ be two clopen sets in $X$ with $\left[ 1_{A}%
\right] -\left[ 1_{B}\right] \in Inf\left( \varphi \right) $. Let $x_{0}\in
A $ and $y_{0}=\varphi ^{w}\left( x_{0}\right) \in B$. Then there is an
element $\pi $ of the full group of $\varphi $ such that $\pi \left(
A\right) =B$ and $\pi \left( x_{0}\right) =y_{0}$ and the function $\zeta
:X\rightarrow \mathbb{Z}^{d}$ satisfying $\pi \left( x\right) =\varphi
^{\zeta \left( x\right) }\left( x\right) $ is continuous on $A\setminus
\left\{ a\right\} $.
\end{theorem}

\begin{proof}
Without loss of generality, $A$ and $B$ are disjoint, otherwise set $\pi =id$
on $A\cap B$. Fix $x_{0}\in A$ and $y_{0}=\varphi ^{w}\left( x_{0}\right)
\in B$. Suppose $\left\{ \epsilon _{n}\right\} $ is a decreasing sequence of
positive numbers which converges to $0$. Set $A_{0}=B_{0}=\emptyset $. Now
for each $n\geq 0$, we recursively do the following

\textbf{Step }$\mathbf{2n+1}$\textbf{:\ }Apply Lemma \ref{step} to $A_{2n}$, 
$B_{2n}$, $x_{0}$, $y_{0}$ and $\epsilon _{2n+1}$. This gives a clopen set $%
\left( A\setminus \cup _{i=0}^{2n}A_{i}\right) \setminus B\left(
x_{0},\epsilon _{2n+1}\right) \subset A_{2n+1}\subset \left( A\setminus \cup
_{i=0}^{2n}A_{i}\right) \setminus \left\{ x_{0}\right\} $ and element $\pi
_{2n+1}$ of the topological full group of $\varphi $ mapping $A_{2n+1}$ into 
$\left( B\setminus \cup _{i=0}^{2n}B_{i}\right) \setminus \left\{
y_{0}\right\} $. Set $B_{2n+1}=\pi _{2n+1}\left( A_{2n+1}\right) $.

\textbf{Step }$\mathbf{2n+2}$\textbf{:\ }Apply Lemma \ref{step} to $%
B\setminus B_{2n+1}$, $A\setminus A_{2n+1}$, $x_{0}$, $y_{0}$ and $\epsilon
_{2n+2}$. This gives a clopen set $B_{2n+2}$ with $\left( B\setminus \cup
_{i=0}^{2n+1}B_{i}\right) \setminus B\left( y_{0},\epsilon _{2n+2}\right)
\subset B_{2n+2}\subset \left( B\setminus \cup _{i=0}^{2n+1}B_{i}\right)
\setminus \left\{ y_{0}\right\} $ and element $\pi _{2n+2}$ of the
topological full group of $\varphi $ mapping $B_{2n+2}$ into $\left(
A\setminus \cup _{i=0}^{2n+1}A_{i}\right) \setminus \left\{ x_{0}\right\} $.
Set $A_{2n+2}=\pi _{2n+2}\left( B_{2n+2}\right) $. (Note that since $\pi
_{2n+2}^{2}=id$, then $\pi _{2n+2}\left( A_{2n+2}\right) =B_{2n+2}$.)

Note that the sets $\left\{ A_{n}\right\} $ and $\left\{ B_{n}\right\} $ are
each pairwise disjoint collections of clopen sets.

Set $\pi \left( x_{0}\right) =y_{0}$ and $\pi =id$ on $X\setminus \left(
A\cup B\right) $. For any other $x\in A$, there will be an $n$ such that $%
\epsilon _{2n+1}$ such that $x\in A\setminus B\left( x_{0},\epsilon
_{2n+1}\right) $. This means that $x\in A_{i}$ for a unique $i$ between $0$
and $2n+1$ which means that $\pi _{i}\left( x\right) $ is not the identity
map for exactly one $i$. Set $\pi \left( x\right) =\pi _{i}\left( x\right)
\in B$ for this value of $i$.

Let us check that $\pi $ is continuous at $x_{0}$ and $y_{0}$, it is fairly
clear that it is continuous everywhere else. Note that the set $A\setminus
\cup _{i=0}^{2n+1}A_{i}$ is a clopen set containing of $x_{0}$ and is a
subset of $B\left( x_{0},\epsilon _{2n+1}\right) $. The image of $A\setminus
\cup _{i=0}^{2n+1}A_{i}$ under $\pi $ is $B\setminus \cup _{i=0}^{2n+1}B_{i}$
which is a clopen set containing $y_{0}$ and 
\begin{equation*}
B\setminus \cup _{i=0}^{2n+1}B_{i}\subset B\setminus \cup
_{i=0}^{2n}B_{i}\subset B\left( y_{0},\epsilon _{2n}\right)
\end{equation*}%
This shows that $\pi $ is continuous at $x_{0}$ and $y_{0}$. That $\pi $ is
one-to-one, onto, and is in the full group is easy to check. Since each $\pi
_{k}$ is in the topological full group, the only possible discontinuity of
the cocycle for $\pi $ is at $a$.
\end{proof}

\begin{theorem}
Let $\left( X,\varphi ,\mathbb{Z}^{d}\right) $ and $\left( Y,\psi ,\mathbb{Z}%
^{d}\right) $ be two minimal Cantor systems which are bounded orbit
injection equivalent, with bounded orbit injections $\theta _{1}$ and $%
\theta _{2}$ into a common system $\left( Z,\alpha ,\mathbb{Z}^{d}\right) $
and let $x_{0}\in X$. Let $h_{\theta _{1}}:\left( G\left( \varphi \right)
,G\left( \varphi \right) _{+}\right) \rightarrow \left( G\left( \alpha
\right) ,G\left( \alpha \right) _{+}\right) $ and $h_{\theta _{2}}:\left(
G\left( \psi \right) ,G\left( \psi \right) _{+}\right) \rightarrow \left(
G\left( \alpha \right) ,G\left( \alpha \right) _{+}\right) $ be the
isomorphisms as in Theorem \ref{thm:iso}. Suppose $h_{\theta _{1}}\left[
1_{X}\right] -h_{\theta _{2}}\left[ 1_{Y}\right] \in Inf\left( \alpha
\right) $. Then there is an orbit equivalence $\theta $ from $\left(
X,\varphi ,\mathbb{Z}^{d}\right) $ to $\left( Y,\psi ,\mathbb{Z}^{d}\right) $
such that the cocycle function $\eta :X\times \mathbb{Z}^{d}\rightarrow 
\mathbb{Z}^{d}$ satisfying%
\begin{equation*}
\varphi ^{w}(x)=y\iff \psi ^{\eta (x,w)}\circ \theta (x)=\theta (y)
\end{equation*}%
has the property that for any $w\in \mathbb{Z}^{d}$, $\eta (\cdot ,w)$ is
continuous on $X\setminus \left\{ x_{0},\varphi ^{-w}x_{0}\right\} $.
\end{theorem}

\begin{proof}
Since $h_{\theta _{1}}\left[ 1_{X}\right] -h_{\theta _{2}}\left[ 1_{Y}\right]
\in Inf\left( \alpha \right) $, we have $\left[ 1_{\theta _{1}\left(
X\right) }\right] -\left[ 1_{\theta _{2}\left( Y\right) }\right] \in
Inf\left( \alpha \right) $. Fix $x_{0}\in X$. By Lemma \ref{fginf}, there is
an element $\pi $ of the full group of $\alpha $ which maps $\theta
_{1}\left( X\right) $ to $\theta _{2}\left( Y\right) $ with the property
that the associated cocycle $\zeta $ satisfying $\pi \left( \theta
_{1}\left( x\right) \right) =\alpha ^{\zeta \left( x\right) }\theta
_{1}\left( x\right) $ is continuous except at $x_{0}$. Set $\theta \left(
x\right) =\theta _{2}^{-1}\pi \theta _{1}\left( x\right) $. Then $\theta $
is an orbit equivalence.

Let $\eta $ be the cocycle function satisfying%
\begin{equation*}
\varphi ^{w}(x)=y\iff \psi ^{\eta (x,w)}\circ \theta (x)=\theta (y)
\end{equation*}%
Fix $w\in \mathbb{Z}^{d}$. We aim to prove that $\eta (\cdot ,w)$ is
continuous except at $x_{0}$ and $\varphi ^{-w}\left( x_{0}\right) $.
Suppose $x\in X\setminus \left\{ x_{0},\varphi ^{-w}\left( x_{0}\right)
\right\} $. Then because $\theta _{1}$ is a bounded orbit injection, there
is a $v\in \mathbb{Z}^{d}$ and a clopen neighborhood $U_{1}$ of $x$ such
that $U_{1}\subset X\setminus \left\{ x_{0},\varphi ^{-w}\left( x_{0}\right)
\right\} $ and $\theta _{1}\varphi ^{w}=\alpha ^{v}\theta _{1}$ on $U_{1}$.
Now since $x\in X\setminus \left\{ x_{0},\varphi ^{-w}\left( x_{0}\right)
\right\} $, there is clopen neighborhood $U_{2}$ of $x$ such that $%
U_{2}\subset U_{1}$ and $\zeta $ is constant on both $U_{2}$ and $\varphi
^{w}U_{2}$. Set $\zeta _{0}=\zeta \theta _{1}\left( x\right) $ and $\zeta
_{1}=\zeta \theta _{1}\varphi ^{w}\left( x\right) $. Then on the set $U_{2}$%
, we have $\alpha ^{\zeta _{1}+v-\zeta _{0}}\pi \theta _{1}=\pi \theta
_{1}\varphi ^{w}$. Finally, because $\theta _{2}$ is a bounded orbit
injection and both $\pi \theta _{1}\left( x\right) $ and $\pi \theta
_{1}\varphi ^{w}\left( x\right) $ are in $\theta _{2}\left( Y\right) $,
there is a clopen neighborhood $U_{3}$ of $x$ such that $U_{3}\subset U_{2}$
and a vector $u$ such that $\theta _{2}^{-1}\alpha ^{\zeta _{1}+v-\zeta
_{0}}\pi \theta _{1}=\psi ^{u}\theta _{2}^{-1}\pi \theta _{1}$. Therefore on 
$U_{3}$, $\theta _{2}^{-1}\pi \theta _{1}\varphi ^{w}=\psi ^{u}\theta
_{2}^{-1}\pi \theta _{1}$, which implies $\eta (\cdot ,w)$ is continuous on $%
X\setminus \left\{ x_{0},\varphi ^{-w}\left( x_{0}\right) \right\} $.
\end{proof}

\providecommand{\bysame}{\leavevmode\hbox to3em{\hrulefill}\thinspace} %
\providecommand{\MR}{\relax\ifhmode\unskip\space\fi MR } 
\providecommand{\MRhref}[2]{  \href{http://www.ams.org/mathscinet-getitem?mr=#1}{#2}
} \providecommand{\href}[2]{#2}

\end{document}